\documentclass[11pt]{amsart}
\usepackage{stmaryrd}
\usepackage{xcolor}
\usepackage{epsfig}
\usepackage{enumerate}
\usepackage{amsmath}
\usepackage{amssymb}
\usepackage{amscd}
\usepackage{graphicx}
\usepackage{lineno}
\usepackage{pstricks}
\usepackage{tocvsec2}
\usepackage{mathtools}
\allowdisplaybreaks[4]

\usepackage [backref, colorlinks, citecolor=red, anchorcolor=black, linkcolor=blue]{hyperref}

\usepackage{amsfonts}
\usepackage[top=35mm, bottom=35mm, left=30mm, right=30mm]{geometry}
\usepackage{verbatim}
\usepackage{graphicx, amssymb, amsmath,amsthm}
\allowdisplaybreaks[4]
\usepackage{appendix}
\usepackage{mathrsfs}
\usepackage{cite}

\makeatletter
\@namedef{subjclassname@2020}{%
	\textup{2020} Mathematics Subject Classification}
\makeatother

\newcommand\SL{\operatorname{SL}}
\newcommand\Sp{\operatorname{Sp}}
\newcommand\GL{\operatorname{GL}}

\newcommand\Supp{\operatorname{Supp}}

\newcommand{\pmat}[4]{\begin{pmatrix} #1 & #2 \\ #3 & #4\end{pmatrix}}
\newcommand{\smallpmat}[4]{\big(\begin{smallmatrix} #1 & #2 \\ #3 & #4\end{smallmatrix}\big)}

\def\Z{\mathbb{Z}}
\def\C{\mathbb{C}}
\def\E{\mathbb{E}}
\def\Q{\mathbb{Q}}
\def\F{\mathbb{F}}

\usepackage{xcolor}

\makeatletter
\NewCommandCopy\@@pmod\pmod
\DeclareRobustCommand{\pmod}{\@ifstar\@pmods\@@pmod}
\def\@pmods#1{\mkern4mu({\operator@font mod}\mkern 6mu#1)}
\makeatother

\usepackage{xcolor}

\newtheorem{theorem}{Theorem}[section]
\newtheorem{definition}[theorem]{Definition}
\newtheorem{lemma}[theorem]{Lemma}

\newtheorem{corollary}[theorem]{Corollary}

\newtheorem{proposition}[theorem]{Proposition}

\numberwithin{equation}{section}

\numberwithin{equation}{section}
\renewcommand*{\backref}[1]{}

\makeatletter

\def\MRbibitem{\@ifnextchar[\my@lbibitem\my@bibitem}

\def\mybiblabel#1#2{\@biblabel{{\hyperref{http://www.ams.org/mathscinet-getitem?mr=#1}{}{}{#2}}}}

\def\myhyperanchor#1{\Hy@raisedlink{\hyper@anchorstart{cite.#1}\hyper@anchorend}}

\def\my@lbibitem[#1]#2#3#4\par{%
	\item[\mybiblabel{#2}{#1}\myhyperanchor{#3}\hfill]#4%
	\@ifundefined{ifbackrefparscan}{}{\BR@backref{#3}}%
	\if@filesw{\let\protect\noexpand\immediate
		\write\@auxout{\string\bibcite{#3}{#1}}}\fi\ignorespaces%
}

\def\my@bibitem#1#2#3\par{%
	\refstepcounter\@listctr
	\item[\mybiblabel{#1}{\the\value\@listctr}\myhyperanchor{#2}\hfill]#3%
	\@ifundefined{ifbackrefparscan}{}{\BR@backref{#2}}%
	\if@filesw\immediate\write\@auxout
	{\string\bibcite{#2}{\the\value\@listctr}}\fi\ignorespaces%
}

\makeatother


\author{Zhaochen Pei}
\address[Zhaochen Pei]
{School of Mathematical Sciences, University of Science and Technology of China, Hefei, Anhui, 230026, P.R. China}
\email{bacon525@mail.ustc.edu.cn}

\author{Maoru Tan}
\address[Maoru Tan]
{School of Mathematical Sciences, University of Science and Technology of China, Hefei, Anhui, 230026, P.R. China}
\email{lwtanmr@mail.ustc.edu.cn}

\author{Weiwen Zhang}
\address[Weiwen Zhang]
{School of Mathematical Sciences, University of Science and Technology of China, Hefei, Anhui, 230026, P.R. China}
\email{llcwwzhang@mail.ustc.edu.cn}

\author{Lilu Zhao}
\address[Lilu Zhao]
{School of Mathematical Sciences, University of Science and Technology of China, Hefei, Anhui, 230026, P.R. China}
\email{zhaolilu@ustc.edu.cn}

\begin{document}

	\begin{abstract}
		Let $Q$ be a translation invariant indefinite quadratic form in $8$ variables. Let $\mathcal{P}$ be the set of all primes. Let $\mathcal{A}\subset\mathcal{P}\cap\{1,\dots,X\}$. By developing the recent work of Green \cite{Green}, we prove that subject to a mild condition on the form $Q$ there exist pairwise distinct primes $p_{1},\dots,p_{8}\in \mathcal{A}$ satisfying $Q(p_{1},\dots,p_{8})=0$ provided that
\[\frac{|\mathcal{A}|}{X/\log X}\gg_{Q}(\log\log X)^{-c}\] 
		for some absolute constant $c>0$.
	\end{abstract}

	
	\title{Translation invariant quadratic forms in $8$ variables and dense sets of primes}
	\maketitle

	\parskip 10pt         

	
\section{Introduction}

For a positive integer $X$, we use the notation $[X]:=\{1,2,\ldots,X\}$. Let $\rho(X)$ denote the maximal density of a subset $\mathcal{A}$ of $[X]$ satisfying that $\mathcal{A}$ contains no nontrivial arithmetic progression of length three. Roth \cite{R1} proved that 
\begin{align}\label{density}\rho(X)\ll (\log\log X)^{-1}.\end{align}
Roth's theorem has been widely and deeply developed. For example, Green \cite{Green2005} obtained a remarkable analogue of Roth's theorem in primes. For the improvement upon the density \eqref{density}, one may refer to \cite{BS,KM} and references therein. 

In recent years, there has been much work on extensions of Roth's theorem to quadratic equations. Browning and Prendiville \cite{BP} studied diagonal quadratic equations over dense sets of integers. Chow \cite{Chow} considered the diagonal equation of degree $d$ over dense sets of primes. Keil \cite{Keil,Keilrefine} considered the analogue of Roth's theorem for non-diagonal translation invariant quadratic forms over dense sets of integers. Keil's result was improved by Zhao \cite{Zhao2019,Zhao2022}. One may also refer to \cite{CLP} and \cite{HH} for related results.

In this paper, we focus on quadratic forms in $8$ variables. We introduce \begin{equation*}
		Q(x,y)=x^{T}ax+x^{T}by+y^{T}cy=(x^{T},y^{T})\begin{pmatrix}
			a &b/2\\
			b^{T}/2 & c
		\end{pmatrix}\begin{pmatrix}
			x\\
			y
		\end{pmatrix},
	\end{equation*}
	where $a,b,c\in \mathrm{Mat}_{4}(\mathbb{Z})$ with $a,c$ symmetric. The following plays an important role
\begin{align}\label{defineDelta}\Delta:=4b^{-1}ab^{-T}c-I_4.\end{align} 
We define \begin{equation}\label{definenu}
		\nu_{Q}(X)=\sum_{\substack{1<x_{1},\dots,x_{8}\leqslant X\\ Q(x_{1},\dots,x_{8})=0}} \Lambda(x_{1})\cdots\Lambda(x_{8}),
	\end{equation}
	where $\Lambda(\cdot)$ is the von Mangoldt function. Very recently, Green \cite{Green} obtained the following impressive result.
	\begin{theorem}[Green \cite{Green}]\label{green}
		Suppose that $\mathrm{det} \,ab \neq0$ and $\Delta$ has
		four distinct eigenvalues which lie in
		$\overline{\mathbb{Q}}\backslash\{-1,0\}$. Then for any
		constant $A>0$, one has
		\[
		\nu_{Q}(X)=\mathfrak{S}_{Q}^\ast X^{6}+O_{A,Q}(X^{6}\log^{-A}X),
		\]
		where $\mathfrak{S}_{Q}^\ast$ is the product of local densities.
	\end{theorem}
Earlier results on this topic were obtained by Liu \cite{Liu} in $10$ variables and by Zhao \cite{Zhao2016} in $9$ variables. The restriction on $Q$ introduced in Theorem \ref{green} is generic in the sense that almost all quadratic forms satisfy the above conditions. Green \cite{Green} pointed out that it seems to be of interest to weaken them as far as possible. We shall first slightly weaken the conditions in Theorem \ref{green}. For convenience, we introduce the following refined generic condition.
	\begin{definition}\label{pan}
		Suppose that $Q(x,y)=x^{T}ax+x^{T}by+y^{T}cy
		$ 
		is a quadratic form in 8 variables. We call $Q$ pan-generic if $\mathrm{det} \,ab \neq0$ and $\Delta=4b^{-1}ab^{-T}c-I_4$ has
		four distinct eigenvalues in
		$\overline{\mathbb{Q}}$, and we denote these two conditions together by \textbf{(PG)}.
	\end{definition} 
Let
\begin{equation*}
		\frac{X}{2}\leqslant Y_{0}<Y_{0}+Y\leqslant X, \qquad Y>\frac{X}{\log X},
	\end{equation*}
and	denote by $I$ the interval
	\begin{equation*}
		I=(Y_{0},Y_{0}+Y].
	\end{equation*}
	Let
	\begin{equation*}
		0\leqslant b<W\leqslant \log X \text{ and } (b,W)=1.
	\end{equation*}
In order to apply Green's $W$-trick, we  introduce
	\begin{equation}\label{definenuW}
		\nu_{Q}(b,W;I)=\sum_{\begin{subarray}{c}
				x_{1},\dots,x_{8}\in I \\ Q(\mathbf{x})=0\\ \mathbf{x}\equiv b\pmod{W}
		\end{subarray}} \Lambda(x_{1})\cdots\Lambda(x_{8}),
	\end{equation}
	where $\mathbf{x}\equiv b\pmod{W}$ means $x_{j}\equiv b\pmod{W}$ for all $1\leqslant j\leqslant 8$. Note that \eqref{definenuW} can be viewed as an extension of \eqref{definenu}. We understand that to capture prime solutions to $Q=0$, one may need to consider longer intervals rather than the dyadic interval. We choose this dyadic interval for convenience and we can also prove that Theorem \ref{extension} holds with $[1,X]$ instead of the dyadic interval $[X/2,X]$.

Now we are able to introduce the following.
	\begin{theorem}
	  \label{extension}
		Suppose that $Q$ is a pan-generic indefinite quadratic form. Then for any constant $A>0$ we have
		\begin{equation}\label{eqextension}
			\nu_{Q}(b,W;I)=\mathfrak{S}^\ast_{Q}(W,b)\mathfrak{K}_Q(I)+O_{A,Q}(Y^{6}W^{2}\phi(W)^{-8}\log^{-A}X),
		\end{equation}
		where $\phi(\cdot)$ is Euler’s totient function, the singular series $\mathfrak{S}^\ast_{Q}(W,b)$ and the singular integral $\mathfrak{K}_Q(I)$ are defined in \eqref{definesingularseries} and \eqref{definesingularint} respectively.
	\end{theorem}
Under the assumption $\det \, ab\neq 0$ we have
	\begin{equation*}
		\begin{pmatrix}
			a &b/2\\
			b^{T}/2 & c
		\end{pmatrix}=\begin{pmatrix}
			1 & 0\\
			b^{T}a^{-T}/2 & 1
		\end{pmatrix}\begin{pmatrix}
			a & 0\\
			0 & b^{T}a^{-1}b\Delta/4
		\end{pmatrix}\begin{pmatrix}
			1 & a^{-1}b/2\\
			0 & 1
		\end{pmatrix}.
	\end{equation*}
The pan-generic condition implies that $\mathrm{rank}(Q)\geqslant 7$. This new generic condition weakens the restriction on $Q$ in Theorem \ref{green} slightly. However, it is crucial for us to consider translation invariant quadratic forms. We say $Q$ is translation invariant if 
$$\begin{pmatrix}
			a &b/2\\
			b^{T}/2 & c
		\end{pmatrix}\mathbf{1}=\mathbf{0}.$$
 We have a direct consequence of Theorem \ref{extension} in the case that $Q$ is translation invariant.
	\begin{corollary}
	  \label{extensioncorollary}
		Suppose that $Q$ is a  pan-generic translation invariant indefinite quadratic form. Then for any constant $A>0$  we have
		\begin{equation*}
			\nu_{Q}(b,W;I)=\mathfrak{S}_{Q}(W)\mathfrak{I}_{Q}\frac{W^{2}Y^{6}}{\phi(W)^{8}}+O_{A,Q}(Y^{6}W^{2}\phi(W)^{-8}\log^{-A}X),
		\end{equation*}
		where  $\mathfrak{S}_{Q}(W)$ and $\mathfrak{I}_{Q}$ are defined in \eqref{definesingularseriesinv} and \eqref{definesingularintinv} respectively. Moreover, there exists a positive number $C_{Q}$ (independent of $W$), such
		that
		\[  \mathfrak{S}_{Q}(W)\mathfrak{I}_{Q}>C_{Q}. \]
	\end{corollary}
Now we state the main result of this paper.
	\begin{theorem}
	\label{Main}
		Suppose that $Q$ is a pan-generic translation invariant indefinite quadratic form. Let $\mathcal{P}$ be the set of all primes. Let $\mathcal{A}\subset\mathcal{P}\cap\{1,\dots,X\}$.  If there exist no pairwise distinct primes $p_{1},\dots,p_{8}\in \mathcal{A}$ such that $Q(p_{1},\dots,p_{8})=0$, then 
\[\frac{|\mathcal{A}|}{X/\log X}\ll_{Q}(\log\log X)^{-c}\] 
		for some absolute constant $c>0$.
	\end{theorem}
In order to prove Theorem \ref{Main}, we need to establish a restriction estimate over minor arcs. 
To introduce our restriction estimate over minor arcs, we first give some notations. We define
	\begin{equation*}
		\Lambda_{b,W;I}(x)=\begin{cases}
			\Lambda(Wx+b) & \text{if }Wx+b\in I,\\
			0& \text{otherwise}.
		\end{cases}
	\end{equation*}
	For any $i\geqslant1$, let $\lambda_{i}:[X]\to\mathbb{C}$ be a sequence satisfying
	\begin{equation}\label{deflambda-i}
		|\lambda_{i}(x)|\leqslant\Lambda_{b,W;I}(x)
	\end{equation}
	for all $x\in [X]$.  We assume for any $i\geqslant 1$ that there exists $0<\delta\leqslant 2$ such that
	\begin{equation*}
		\sum\limits_{x}|\lambda_{i}(x)|\leqslant\frac{\delta Y}{\phi(W)}.
	\end{equation*}
	Now we consider the general exponential sum 
\begin{equation}\label{defineStheta}
		S(\theta):=S(\theta;\lambda_{1},\dots,\lambda_{8})=\sum\limits_{x_{1},\dots,x_{8}} \lambda_{1}(x_{1})\cdots\lambda_{8}(x_{8})e(\theta Q(x_{1},\dots,x_{8})).
	\end{equation}

Let $M$ be a positive number to be fixed, and set
	\begin{equation}\label{defineM'}
		M'= M^{10^{4}c_{0}^{-1}}
	\end{equation}
	for some positive constant $c_{0}$ which is exactly the exponent appearing in Proposition \ref{proposition51}. 
	Set
	\begin{equation*}
		K=8\max\limits_{ij} |b_{ij}|,
	\end{equation*}
	where the $b_{ij}$ are the entries of $b$. Then $K$ depends only on $Q$. For each $q\in\mathbb{N}$ and $r\in(\mathbb{Z}/q\mathbb{Z})^{*}$, denote
	\begin{equation*}
		I_{r,q}:=\left\{\theta\in\mathbb{T}:|\theta-\frac{r}{q}|\leqslant \frac{Mq^{c_{1}}}{(Y/W)^{2}}\right\}
	\end{equation*} and
	\begin{equation*}
		\tilde{I}_{r,q}:=\left\{\theta\in\mathbb{T}:|\theta-\frac{r}{q}|\leqslant \frac{1}{Kq(Y/W)}\right\},
	\end{equation*}
	where 
\begin{align}\label{definec1}c_{1}=\frac{c_{0}}{10000}.\end{align} Now we define the \textit{major arcs}
	\begin{equation}\label{eq:definefrakM}
		\mathfrak{M}(M'):=\bigcup\limits_{\substack{1\leqslant q\leqslant M'\\ r\in(\mathbb{Z}/q\mathbb{Z})^{*}}} I_{r,q} 
	\end{equation}
	and set
	\begin{equation*}
		\mathfrak{m}_{1}(M'):=\bigcup\limits_{\substack{1\leqslant q\leqslant KY/W\\ r\in(\mathbb{Z}/q\mathbb{Z})^{*}}}  (\tilde{I}_{r,q}\backslash I_{r,q}),\quad\mathfrak{m}_{2}(M'):=\bigcup\limits_{\substack{ M'< q\leqslant KY/W\\ r\in(\mathbb{Z}/q\mathbb{Z})^{*}}} I_{r,q}.
	\end{equation*}
	We define $\mathfrak{m}(M')=\mathfrak{m}_{1}(M')\cup\mathfrak{m}_{2}(M')$ to be the \textit{minor arcs}. Note that $I_{r,q}$ in \eqref{eq:definefrakM} are pairwise disjoint when $2M(M')^{2+c_1}<(Y/W)^2$. By Dirichlet’s theorem on diophantine approximation,
\begin{equation*}
		\mathbb{T}=\bigcup\limits_{\substack{1\leqslant q\leqslant KY/W\\ r\in(\mathbb{Z}/q\mathbb{Z})^{*}}}\tilde{I}_{r,q},
	\end{equation*} 
thus we do have $\mathbb{T}=\mathfrak{M}(M')\cup\mathfrak{m}(M')=\mathfrak{M}(M')\cup\mathfrak{m}_{1}(M')\cup\mathfrak{m}_{2}(M')$.

	\begin{proposition}
	  [Restriction estimate on $\mathfrak{m}$]\label{propminor}
		Suppose that $Q$ is a pan-generic quadratic form. Suppose that $M\leqslant \log^{C}X$ for some $C\geqslant10$. Then we have
		\begin{equation*}
			\int_{\mathfrak{m}(M')}|S(\theta)|d\theta \ll_{Q} \frac{W^{2}Y^{6}}{\phi(W)^{8}}(M')^{-10^{-4}c_{0}}.
		\end{equation*}
	\end{proposition}

\section{Proofs of Theorem \ref{extension} and Corollary \ref{extensioncorollary}}

We first introduce some definitions and present some estimates of the singular series and singular integrals. All these are basically quoted from \cite{Zhao2022} (and also standard). In fact, all results in this section hold for quadratic forms with rank at least five, and this will be assumed throughout. 

Define 
	\begin{equation*}
		S_{W,b}^{*}(q,r)=\sum_{\begin{subarray}{c}
				1\leqslant\mathbf{x}\leqslant qW \\ (\mathbf{x},q)=1\\ \mathbf{x}\equiv b(\text{mod } W)
		\end{subarray}} e_{q}(rQ(\mathbf{x})),
	\end{equation*}
and
	\begin{equation*}
		B_{W,b}^{*}(q)=\frac{1}{\phi(qW)^{8}}\sum_{r\in(\mathbb{Z}/q\mathbb{Z})^{*}}	S_{W,b}^{*}(q,r).
	\end{equation*}
Suppose that $(b,W)=1$, and we have
		\begin{equation*}
			B_{W,b}^{*}(q)\ll\frac{W^{5}}{\phi(W)^{8}}q^{-\frac{3}{2}+\varepsilon}.
		\end{equation*}
Then we can define the singular series 
	\begin{equation}\label{definesingularseries}
		\mathfrak{S}^{*}_{Q}(W,b)=\sum_{q=1}^{\infty}B_{W,b}^{*}(q).
	\end{equation}

	Now we turn to the singular integrals, which are well studied in the asymptotic formula for the integer solutions to $Q(\mathbf{x})=0$. One may refer to \cite[Section 3]{Keil} and \cite[Chapter 8]{Davenport} for more details. Define
	\begin{equation*}
		\mathcal{K}(\eta;I)=\int_{I^{8}} e(\eta Q(\mathbf{x}))d\mathbf{x}.
	\end{equation*}
	Subject to the condition $\mathrm{rank}(Q)\geqslant 5$, we have
	\begin{equation*}
		\mathcal{K}(\eta;I)\ll Y^{8}(1+Y^{2}|\eta|)^{-2}.
	\end{equation*}
	We introduce the singular integral
	\begin{equation}\label{definesingularint}
		\mathfrak{K}_{Q}(I)=\int_{-\infty}^{\infty}\mathcal{K}(\eta;I)d\eta
	\end{equation}
	and
	\begin{equation}\label{definesingularintinv}
		\mathfrak{I}_Q=\int_{-\infty}^{\infty}\left(\int_{(0,1]^{8}} e(\eta Q(\mathbf{x}))d\mathbf{x}\right)d\eta.
	\end{equation}
When $Q$ is translation invariant, we are able to explain more on the singular series and singular integral.
Let \begin{equation*}
		S(q,r)=\sum_{1\leqslant\mathbf{x}\leqslant q}
		e_{q}(rQ(\mathbf{x})), \qquad
		B(q)=\frac{1}{q^{8}}\sum_{r\in(\mathbb{Z}/q\mathbb{Z})^{*}}	S(q,r).
	\end{equation*}
	For each prime $p$, we define the local densities
	\begin{equation*}
		\sigma^{*}_{p}=\sum_{k=0}^{\infty}B^{*}_{1,0}(p^{k}), \qquad
		\sigma_{p}=\sum_{k=0}^{\infty}B(p^{k}),
	\end{equation*}
	and their product
	\begin{equation}\label{definesingularseriesinv}
		\mathfrak{S}_{Q}(W)=\left(\prod_{p|W}\sigma_{p}\right)\left(\prod_{p\nmid W}\sigma^{*}_{p}\right).
	\end{equation}
	\begin{lemma}[Lemma 2.7 \cite{Zhao2022}]\label{lemma31}
		Suppose that $Q$ is translation invariant. Then we have
		\begin{equation*}
			\mathfrak{S}^{*}_{Q}(W,b)=\frac{W^{2}}{\phi(W)^{8}}	\mathfrak{S}_{Q}(W).
		\end{equation*}
		Moreover, there exists some positive constant $C_Q$ independent of $W$ such that
		\begin{equation*}
			\mathfrak{S}_{Q}(W)> C_Q.
		\end{equation*}
     \end{lemma}
We also have the following.
	\begin{lemma}[Lemma 2.8 \cite{Zhao2022}]\label{lemma32}
		Suppose that $Q$ is translation invariant. Then we have
		\begin{equation*}
			\mathfrak{K}_{Q}(I)=Y^{6}\mathfrak{I}_{Q}.
		\end{equation*}
		Moreover, if $Q$ is indefinite then $\mathfrak{I}_{Q}>0$.
	\end{lemma}

Define
\begin{align*}
		T(\theta)=\sum_{x_{1},\ldots,\,x_{8}\in I}\xi(x_1)\cdots\xi(x_8) e(\theta Q(x_{1},\dots,x_{8})),
	\end{align*}
where the weight satisfies $|\xi(x)|\leqslant \Lambda(x)$. Taking $W=1,b=0$ in Proposition \ref{propminor}, we have
		\begin{equation}\label{boundTminor}
			\int_{\mathfrak{m}(M')}|T(\theta)|d\theta \ll Y^{6}M^{-1}.
		\end{equation}
We choose $$\xi(x)=\chi_{b,W;I}(x)\Lambda(x),$$
 where $\chi_{b,W;I}(x)$ is the characteristic function for $\{x\in I:\ x\equiv b\pmod{W}\}$. Then we have
	\begin{equation*}
			\nu_{Q}(b,W;I)=\int_0^1T(\theta)d\theta.
		\end{equation*}

	\begin{lemma}
	  	  [Major arcs]\label{major}
		Suppose that $Q$ is an indefinite quadratic form in 8 variables, with $\mathrm{rank}(Q)\geqslant 5$. If $M=\log^{B}X$ for some $B\geqslant A+10$, then we have
		\begin{equation*}
			\int_{\mathfrak{M}(M')}T(\theta)d\theta=\mathfrak{S}^{*}_{Q}(W,b)\mathfrak{K}_{Q}(I)+O_{A,Q}(Y^{6}W^{2}\phi(W)^{-8}\log^{-A}X).
		\end{equation*}
	\end{lemma}
\begin{proof}The proof can be carried out by completely following the proof of (3.11) in \cite{Zhao2022}.\end{proof}
\begin{proof}[Proofs of Theorem \ref{extension} and Corollary \ref{extensioncorollary}  assuming Proposition \ref{propminor}.] Take $M=\log^{B}X$ for some constant $B\geqslant A+10$. Since $W\leqslant \log X$, by \eqref{boundTminor} (which needs Proposition \ref{propminor}), 		
         \begin{equation}\label{boundTminorlog}
			\int_{\mathfrak{m}(M')}|T(\theta)|d\theta \ll _{A,Q} Y^{6}W^{2}\phi(W)^{-8}\log^{-A}X.
		\end{equation}
Now \eqref{boundTminorlog} and Lemma \ref{major} together complete the proof of Theorem \ref{extension}. 

Corollary \ref{extensioncorollary} follows from Theorem \ref{extension}, Lemma \ref{lemma31} and Lemma \ref{lemma32}. \end{proof}

\bigskip

\section{The contribution from integration over $\mathfrak{m}_1$}

\begin{lemma}\label{lemmastand}Let $\theta=\frac{r}{q}+\eta\in \tilde{I}_{r,q}$. Suppose that $F_1,F_2:\, [X]^4\rightarrow \mathbb{C}$ are $1$-bounded functions, namely $\|F_i\|_{\infty}\leqslant 1$ ($i=1,2$), and supported on the set $\{x:Wx+b\in I\}$. Then we have
\begin{align}\label{standard}\sum\limits_{x,\,y} F_{1}(x)F_{2}(y)e(\theta Q(x,y)) \ll X^6\log^2X+\frac{X^4\log^2 X}{\eta^2q^2}.\end{align}
\end{lemma}
\begin{proof}The proof follows from the argument in Section 5 of \cite{Green}.\end{proof}
We give some remarks on the estimate \eqref{standard}. Roughly speaking, the $\eta$-aspect is good, however the $q$-aspect is {\it bad}. Therefore, whenever $\eta$ is a little bit large, say $|\eta|>\frac{M}{X^2}$, one can deduce a nice mean value estimate with upper bound $O(M^{-1}X^6\log^3X)$. The  
hardest thing is to improve the $q$-aspect (subject to the condition $|\eta|\leqslant \frac{M}{X^2}$), and for this purpose Green \cite{Green} introduced the group representation to this problem. 

In order to prove Theorem \ref{Main}, the upper bound $O(M^{-1}X^6\log^3X)$ is insufficient since $M$ is possibly a constant. To obtain a $\log$-free upper bound, one may appeal to the sieve method. This has been explained in \cite{Zhao2022}, and we summarize the following conclusion.   
\begin{lemma}\label{lemmasieve}
Let $A_0,B_0$ be two constants. Let $\theta=\frac{r}{q}+\eta$ with 
$$(r,q)=1, \quad 1\leqslant q\leqslant (\log X)^{A_0},\ \quad |\eta|\leqslant \frac{q^{c_1}(\log X)^{B_0}}{X^2}.$$ Then we have 
\begin{align}\label{inequalitysieve}|S(\theta)|\ll \frac{Y^8q^{o(1)}}{\phi(W)^8\phi(q)^2}\Big(1+Y^2W^{-2}|\eta|\Big)^{-2}.\end{align}
\end{lemma}
\begin{proof} Let $f(x,y)$ be a quadratic form in $2s$ variables given by
$$f(x,y)=x^Ta_0x+x^Tb_0y+y^Tc_0y=(x^{T},y^{T})\begin{pmatrix}
			a_0 &b_0/2\\
			b_0^{T}/2 & c_0
		\end{pmatrix}\begin{pmatrix}
			x\\
			y
		\end{pmatrix},$$ 
where $a_0,b_0,c_0\in \mathrm{Mat}_{s}(\mathbb{Z})$ with $a_0,c_0$ symmetric and $b_0$ invertible. Then it was essentially proved in \cite{Zhao2022} (see Sections 4-5 in \cite{Zhao2022}) that
$$\sum_{x_{1},\dots,x_{2s}} \lambda_{1}(x_{1})\cdots\lambda_{2s}(x_{2s})e(\theta f(x_{1},\dots,x_{2s}))\ll 
 \frac{Y^{2s}\tau(q)^{\frac{s+1}{2}}}{\phi(W)^{2s}\phi(q)^{s/2}}\Big(1+Y^2W^{-2}|\eta|\Big)^{-s/2},$$
 where we recall that $\lambda_i$ is defined in \eqref{deflambda-i} and $\tau(q)$ is the divisor function. We remark that although the above result is proved in the case $s=5$, the proof works well for all $s$ without any changes. 
 
Taking $s=4$, we obtain \eqref{inequalitysieve}. This completes the proof.
\end{proof}

Note that \eqref{inequalitysieve} can be viewed as a log-free refinement of \eqref{standard}, and it was proved by invoking the sieve method. Again, the $\eta$-aspect is good enough, but the $q$-aspect is insufficient. Since $\eta$-aspect is large when $\theta=\frac{r}{q}+\eta\in \mathfrak{m}_{1}(M')$, now we are able to establish the following. 
\begin{lemma}
  [Minor arcs $\mathfrak{m}_{1}$]\label{m1}
		Suppose that $\det b\neq 0$. Suppose that $M\leqslant \log^{C}X$ for some constant $C\geqslant10$. Then we have
		\begin{equation}
			\int_{\mathfrak{m}_{1}(M')}|S(\theta)|d\theta \ll_{Q}\frac{W^{2}Y^{6}}{\phi(W)^{8}}M^{-1} .
		\end{equation}
\end{lemma}
\begin{proof} Note that $\tilde{I}_{r,q}\backslash I_{r,q}$ is empty if $q>(\frac{Y}{KWM})^{\frac{1}{1+c_1}}:=T_0$. Let $A_0$ and $B_0$ be two constants to be chosen later. We introduce
	\begin{equation*}
		I^{0}_{r,q}=\left\{\theta\in\mathbb{T}:|\theta-\frac{r}{q}|\leqslant \min\Big(\frac{Mq^{c_{1}}(\log X)^{B_0}}{(Y/W)^{2}},\,\frac{1}{Kq(Y/W)}\Big)\right\}.
	\end{equation*}
Then we define
$$\mathfrak{r}_1=\bigcup\limits_{\substack{1\leqslant q\leqslant T_0\\ r\in(\mathbb{Z}/q\mathbb{Z})^{*}}}  (\tilde{I}_{r,q}\setminus I^0_{r,q}), \qquad \mathfrak{r}_2=\bigcup\limits_{\substack{(\log X)^{A_0}< q\leqslant T_0\\ r\in(\mathbb{Z}/q\mathbb{Z})^{*}}}  (I^0_{r,q}\setminus I_{r,q})$$
and
$$\mathfrak{r}_3=\bigcup\limits_{\substack{q\leqslant (\log X)^{A_0}\\ r\in(\mathbb{Z}/q\mathbb{Z})^{*}}}  (I^0_{r,q}\setminus I_{r,q}).$$
Note that 
$$\mathfrak{m}_{1}(M')=\mathfrak{r}_1\cup \mathfrak{r}_2\cup\mathfrak{r}_3.$$
By Lemma \ref{lemmastand}, we have
\begin{align*}S(\theta) \ll X^6\log^{10}X+\frac{X^4\log^{10} X}{\eta^2q^2}.\end{align*}
Then we can deduce that
		\begin{align*}
			\int_{\mathfrak{r}_1}|S(\theta)|d\theta \ll &\,\sum_{1\leqslant q\leqslant T_0}\phi(q)\Big(\frac{1}{Kq(Y/W)} X^6\log^{10}X+\frac{X^4\log^{10} X}{q^2}
\frac{(Y/W)^{2}}{Mq^{c_{1}}(\log X)^{B_0}}\Big)
\\ \ll &\, (Y/W)^{-\frac{c_1}{1+c_1}}M^{-\frac{1}{1+c_1}}K^{-\frac{2+c_1}{1+c_1}}X^6\log^{10}X+\frac{(Y/W)^2X^4\log^{10}X}{M(\log X)^{B_0}},
		\end{align*}
and
		\begin{align*}
			\int_{\mathfrak{r}_2}|S(\theta)|d\theta \ll &\,
\sum_{(\log X)^{A_0}< q\leqslant T_0}\phi(q)\Big(\frac{1}{Kq(Y/W)} X^6\log^{10}X+\frac{X^4\log^{10} X}{q^2}
\frac{(Y/W)^{2}}{Mq^{c_{1}}}\Big)
\\ \ll &\, (Y/W)^{-\frac{c_1}{1+c_1}}M^{-\frac{1}{1+c_1}}K^{-\frac{2+c_1}{1+c_1}}X^6\log^{10}X+\frac{(Y/W)^2X^4\log^{10}X}{M(\log X)^{A_0c_1}}.
		\end{align*}
Recall that $W\leqslant \log X$ and $Y> \frac{X}{\log X}$, by choosing $A_0$ and $B_0$ large (say $A_0>30/c_1$ and  $B_0>30$), we obtain
\begin{align*}
			\int_{\mathfrak{r}_1\cup\mathfrak{r}_2}|S(\theta)|d\theta \ll \frac{W^{2}Y^{6}}{\phi(W)^{8}}M^{-1}.
		\end{align*}

For the integration over $\mathfrak{r}_3$, we apply Lemma \ref{lemmasieve} to deduce that
\begin{align*}
			\int_{\mathfrak{r}_3}|S(\theta)|d\theta \ll &\,
\sum_{q\leqslant (\log X)^{A_0}}\phi(q)\int_{\frac{Mq^{c_{1}}}{(Y/W)^{2}}}^{+\infty}\frac{Y^8q^{\varepsilon}}{\phi(W)^8\phi(q)^2}\Big(1+Y^2W^{-2}|\eta|\Big)^{-2}d\eta
\\ \ll&\,\frac{W^{2}Y^{6}}{\phi(W)^{8}}M^{-1}.
		\end{align*}
The proof of this lemma is complete.
\end{proof}

\section{Restricted Gauss sums on average}

Let $q$ be a positive integer, and let $f_1,f_2:(\mathbb{Z}/q\mathbb{Z})^4\rightarrow \mathbb{C}$. Following \cite{Green}, we introduce the following restricted Gauss sums on average
\begin{align}\label{defineT}
T_{f_1,f_2}(r):=q^2\mathbb{E}_{x,y\in (\mathbb{Z}/q\mathbb{Z})^4}f_1(x)f_2(y)e\Big(\frac{r}{q}Q(x,y)\Big).\end{align}
The dependence on $q$ in the above notation is omitted. For convenience, we also define
\begin{align}\label{defineScal}
\mathcal{S}_{f_1,f_2}(q,r)=\sum_{x,y\in (\mathbb{Z}/q\mathbb{Z})^4}f_1(x)f_2(y)e\Big(\frac{r}{q}Q(x,y)\Big).\end{align}
One has $T_{f_1,f_2}(r)=q^{-6}\mathcal{S}_{f_1,f_2}(q,r)$. The crucial step is to prove the following, which is an extension of \cite[Proposition 6.2]{Green}.

\begin{proposition}
   \label{proposition51}
		There is an absolute constant $c_0 > 0$ with the following property. Suppose that $Q$ is pan-generic. There is $p_0(Q)$ such that if $q$ is squarefree and with all prime factors greater than $p_0(Q)$, then we have
\begin{equation}\label{zq-main-ineq}
			\frac{1}{\phi(q)}\sum_{r \in (\mathbb{Z}/q\mathbb{Z})^*} |{T_{f_1, f_2}(r)}| \ll q^{- c_0} \| f_1 \|_2 \| f_2 \|_2
		\end{equation}
		for any $f_1, f_2 \in \ell^2((\mathbb{Z}/q\mathbb{Z})^4)$. 
\end{proposition}

The purpose of this section is to prove Proposition \ref{propminor} by assuming Proposition \ref{proposition51}. And the main task of the remaining sections is to prove Proposition \ref{proposition51}. 

We first recall some results from \cite{Green}.
\begin{lemma}
  [Lemma 6.1 in \cite{Green}]\label{lemma52}
		Suppose that $f_{1},f_{2}:(\mathbb{Z}/q\mathbb{Z})^{4}\to \mathbb{C}$ and $\det b\neq0$. Then for any $r\in(\mathbb{Z}/q\mathbb{Z})^{*}$ we have
		\begin{equation*}
			|T_{f_{1},f_{2}}(r)|\leqslant (q,\det b)^{2}\|f_{1}\|_{2}\|f_{2}\|_{2}.
		\end{equation*}
\end{lemma}

\begin{lemma}\label{lemma53}Recall \eqref{defineScal}.
		Suppose that $q=q_1q_0$ with $(q_0,q_1)=1$. Then
		\begin{equation}\label{reductiontoq0}
			\sup_{f_1,f_2}\sum_{r\in (\mathbb{Z}/q\mathbb{Z})^\ast}|\mathcal{S}_{f_1,f_2}(q,r)|\leqslant q_1^{17}\sup_{g_1,g_2}\sum_{r\in (\mathbb{Z}/q_0\mathbb{Z})^\ast}|\mathcal{S}_{g_1,g_2}(q_0,r)|,
		\end{equation}
where the $\sup$ is over $1$-bounded functions $f_1,f_2:(\mathbb{Z}/q\mathbb{Z})^4\rightarrow \mathbb{C}$ and $g_1,g_2:(\mathbb{Z}/q_0\mathbb{Z})^4\rightarrow \mathbb{C}$ respectively.
\end{lemma}
\begin{proof}The desired result directly follows from taking $X=q$ in \cite[Lemma 6.4]{Green}. Indeed, \eqref{reductiontoq0} holds with $q_1^{17}$ replaced by $q_1^9$ if we follow the proof of \cite[Lemma 6.4]{Green}.\end{proof}

	Let $p_{0}(Q)$ be the threshold in Proposition \ref{proposition51}. For any $q$, we write $q=q_0q_1$, where $q_{0}$ for the product of all primes $p>p_{0}(Q)$ which
	divide $q$ precisely once, and $q_{1}$ is the product of all prime powers
	$p^{j}\parallel q$ with $p\leqslant p_{0}$ or $j\geqslant 2$. Note that $(q_{0},q_{1})=1$. Then one has, see \cite[(6.16)]{Green}, that
		\begin{equation}\label{sumoverq}
		\sum_{H<q\leqslant 2H}\min\{q^{-1},\ q_{1}^{11}q^{-1-c_{0}}\}\ll H^{-c_{0}/100}.
		\end{equation}

Let $A_0>0$ be a large constant to be chosen later. We introduce 
\begin{align}\label{defineBq}
B_q:=B_q(I,W,b)=\begin{cases}
\frac{Y}{\phi(qW)} \ \ &\, q\leqslant (\log X)^{A_0},
\\ 
\frac{Y\log X}{qW} \ \ &\, (\log X)^{A_0}<q\leqslant KY/W.\end{cases}
\end{align}
Recall \eqref{deflambda-i}. Then for any $u$, we have	\begin{equation}\label{divideintomodq}
\sum_{x\equiv u\pmod{q}}\lambda_i(x)\ll \sum_{\substack{y\in I \\ y\equiv Wu+b\pmod{qW}}}\Lambda(y)\ll B_q.
		\end{equation}
Note that $O(\frac{Y\log X}{qW})$ is the trivial upper bound for all $q\leqslant KY/W$. When $ q\leqslant (\log X)^{A_0}$, one can use the Siegel-Walfisz theorem to improve the trivial bound to $O(\frac{Y}{\phi(qW)})$.

Now we are ready to prove the following.
\begin{lemma}
  [Minor arcs $\mathfrak{m}_{2}$]\label{m2}
		Suppose that $Q$ is a pan-generic quadratic form. Suppose that $M\leqslant \log^{C}X$ for some constant $C\geqslant10$. Then we have
		\begin{equation}\label{estimateinlemmam2}
			\int_{\mathfrak{m}_{2}(M')}|S(\theta)|d\theta \ll_{Q}\frac{W^{2}Y^{6}}{\phi(W)^{8}}M^{-1} .
		\end{equation}
\end{lemma} 
\begin{proof} We introduce the notation 
	\begin{equation*}
		S_{F_1,F_2}(\theta)=\sum\limits_{x,y} F_1(x)F_2(y)e(\theta Q(x,y))
	\end{equation*}
and	\begin{equation*}
		S_{F}(\theta)=\sum\limits_{\mathbf{x}} F(\mathbf{x})e(\theta Q(\mathbf{x})),
	\end{equation*}
where $F_1,F_2:[X]^4\rightarrow \mathbb{C}$ and $F:[X]^8\rightarrow \mathbb{C}$.

Let $\omega:\mathbb{R}^{8}\to\mathbb{R}$ be a smooth compactly-supported function with $\omega(\mathbf{t})=1$ for $\mathbf{t}\in [0,1]^{8}$. Set for any real parameter $\lambda$,\begin{equation*}
		\Omega_{\lambda}(\mathbf{t}):=\omega(\mathbf{t}-(Y_{0}-b)/Y)e(\lambda Q(\mathbf{t})).
	\end{equation*}
Now taking $F(\mathbf{x})=\prod_{i=1}^8\lambda_i(x_i)$, we have
	\begin{equation}\label{installW}
		S(\frac{r}{q}+\eta)=S_{F}(\frac{r}{q}+\eta)=\sum\limits_{\mathbf{x}} F(\mathbf{x})e_{q}(rQ(\mathbf{x}))\Omega_{\eta (Y/W)^{2}}(\frac{\mathbf{x}}{Y/W}).
	\end{equation}
	Integration by parts gives
	\begin{equation}\label{boundW}
		\|\hat{\Omega}_{\lambda}\|_{1}\ll \max(1,|\lambda|)^{16}.
	\end{equation}
	By the Fourier inversion, from \eqref{installW} we deduce that
	\begin{equation}
		S(\frac{r}{q}+\eta)=\int_{\mathbb{R}^{8}}\hat{\Omega}_{\eta (Y/W)^{2}}(\xi)S_{F_{\xi}}(\frac{r}{q})d\xi,
	\end{equation}
	where $F_{\xi}(\mathbf{x}):=F(\mathbf{x})e\big(\xi\cdot\mathbf{x}/(Y/W)\big)$. Therefore,
	\begin{align*}
		\int_{\mathfrak{m}_{2}(M')}|S(\theta)|d\theta&\leqslant	\sum_{M'<q\leqslant K(Y/W)}\int_{|\eta|\leqslant \frac{Mq^{c_{1}}}{(Y/W)^{2}}}\sum_{r \in (\mathbb{Z}/q\mathbb{Z})^\ast}|	S_{F}(\frac{r}{q}+\eta)|d\eta\notag\\
		&\leqslant	\sum_{M'<q\leqslant K(Y/W)}\int_{|\eta|\leqslant \frac{Mq^{c_{1}}}{(Y/W)^{2}}} \int_{\mathbb{R}^{8}}	|\hat{\Omega}_{\eta (Y/W)^{2}}(\xi)|\sum_{r \in (\mathbb{Z}/q\mathbb{Z})^\ast}|S_{F_{\xi}}(\frac{r}{q})|d\xi d\eta.
	\end{align*}
and then by \eqref{boundW} and the dyadic argument, we further have
	\begin{equation}\label{divideintoH0}
			\int_{\mathfrak{m}_{2}(M')}|S(\theta)|d\theta \ll \sum_{\substack{H=2^j~(j\in \mathbb{Z}^{+})\\ H> M'/2}}\frac{M^{17}H^{17c_1}}{(Y/W)^2}\sup_{\xi}\mathcal{H}(\xi),
		\end{equation}
where
\begin{align}\label{definecalH}\mathcal{H}(\xi):=\mathcal{H}(\xi;H)=\sum_{H<q\leqslant 2H}\sum_{r \in (\mathbb{Z}/q\mathbb{Z})^\ast}|S_{F_\xi}(\frac{r}{q})|.\end{align}

Write $G_1(x)=\prod_{i=1}^4\lambda_{i}(x_{i})e\big(\xi_ix_i/(Y/W)\big)$ and $G_2(y)=\prod_{i=5}^8\lambda_{i}(x_{i})e\big(\xi_ix_i/(Y/W)\big)$. 
Recalling the notation \eqref{defineScal}, we deduce that
	\begin{align*}
		\sum_{r \in (\mathbb{Z}/q\mathbb{Z})^\ast}|S_{F_\xi}(\frac{r}{q})|=&\, B_q^8\sum_{r \in (\mathbb{Z}/q\mathbb{Z})^\ast}|\sum_{u,v\pmod{q}}e(\frac{r}{q}Q(u,v))f_1(u)f_2(v)|
\\ = &\, B_q^8\sum_{r \in (\mathbb{Z}/q\mathbb{Z})^\ast}|\mathcal{S}_{f_1,f_2}(q,r)|,
	\end{align*}
where $f_i(u)=\frac{1}{B_q^4}\sum_{x\equiv u\pmod{q}}G_i(x)$ satisfies $f_i(u)\ll 1$ by \eqref{divideintomodq}. By Lemma \ref{lemma52}, Proposition \ref{proposition51} and Lemma \ref{lemma53}, one has
\begin{align}\label{boundgausssum}
		\sum_{r \in (\mathbb{Z}/q\mathbb{Z})^\ast}|S_{F_\xi}(\frac{r}{q})|\ll B_q^8 \min(q^7,\ q_1^{17}q_0^{7-c_0}).
	\end{align}

Set $\mathcal{L}=(\log X)^{A_0}$. When $H\leqslant \mathcal{L}/2$, inserting \eqref{boundgausssum} into \eqref{definecalH}, we have by \eqref{sumoverq} that
\begin{align*}
\mathcal{H}(\xi)\ll \frac{Y^8H^{o(1)}}{\phi(W)^8}\sum_{H<q\leqslant 2H}\min(q^{-1},\ q_1^{11}q^{-1-c_0})\ll \frac{Y^8H^{-c_0/100+o(1)}}{\phi(W)^8}.\end{align*} Then recalling $c_1=c_0/10000$ and \eqref{defineM'}, we further have
	\begin{equation}\label{divideintoH1}
\sum_{\substack{H=2^j\\ M'/2<H\leqslant \mathcal{L}/2}}\frac{M^{17}H^{17c_1}}{(Y/W)^2}\sup_{\xi}\mathcal{H}(\xi)\ll \frac{Y^8}{(Y/W)^2\phi(W)^8}\cdot \frac{M^{17}}{M'^{\frac{c_0}{200}}}\ll \frac{Y^8}{(Y/W)^2\phi(W)^8}M^{-1}.
		\end{equation}
When $H> \mathcal{L}/2$, there is an extra factor $\frac{\phi(qW)}{qW}\log X$ in $B_q$. In this case, we have
\begin{align*}
\mathcal{H}(\xi)\ll  \frac{Y^8H^{-c_0/101}(\log X)^{8}}{\phi(W)^8}\end{align*}
and then by choosing $A_0=200(18C+9)/c_0$
 	\begin{equation}\label{divideintoH2}
\sum_{\substack{H=2^j\\ H> \mathcal{L}/2}}\frac{M^{17}H^{17c_1}}{(Y/W)^2}\sup_{\xi}\mathcal{H}(\xi)\ll \frac{Y^8}{(Y/W)^2\phi(W)^8}\cdot \frac{M^{17}(\log X)^{8+o(1)}}{\mathcal{L}^{c_0/200}}\ll \frac{Y^8}{(Y/W)^2\phi(W)^8}M^{-1}.
		\end{equation}
Now, \eqref{estimateinlemmam2} follows from \eqref{divideintoH0}, \eqref{divideintoH1} and \eqref{divideintoH2}.
\end{proof}

\begin{proof}[Proof of Proposition \ref{propminor} assuming Proposition \ref{proposition51}.] In fact, Proposition \ref{propminor} follows from Lemma \ref{m1} and Lemma \ref{m2} immediately. Note that Proposition \ref{proposition51} is used in the proof of Lemma \ref{m2}.\end{proof}

\bigskip

\section{Green's interpretation of exponential sums by Weil representation}

In this section, we review Green's interpretation of exponential sums as matrix coefficients by using the Weil representation.
	We begin with the definition for symplectic group.
	\begin{definition}
		We define the symplectic group $\Sp_8(\Z /q\Z)$ to be the group consisting of all $8 \times 8$ block matrices $g = \smallpmat{A}{B}{C}{D}$ with entries in $\Z /q\Z$ and $A^T C = C^T A$, $B^T D = D^T B$ and $A^T D - C^T B = I$.
	\end{definition}

Let $\operatorname{U}(\ell^2((\Z/q\Z)^4))$ be the group of all unitary operators on $\ell^2((\Z/q\Z)^4)$. The following statement is \cite[Proposition 7.2]{Green}, see also its remarks in \cite{Green}.

\begin{proposition}[Weil representation]\label{weil-prop}
		Let $q$ be squarefree and odd. Then there is a unitary representation \[ \rho : \Sp_8(\Z/q\Z) \rightarrow \operatorname{U}(\ell^2((\Z/q\Z)^4))\] and a function $\xi : \Sp_8(\Z/q\Z) \rightarrow \{z \in \C : |z| = 1\}$ satisfying the following:
		\begin{itemize} 
			\item \textup{(Dilations)} If $g = s(E)$ where $s(E) := \smallpmat{E}{0}{0}{E^{-T}}$ with $E$ invertible then
			\[ \rho(g) f(x) = \xi(s(E)) f(E^{-1} x);\]
			\item \textup{(Fourier transform)} If $g = J:= \smallpmat{0}{I}{-I}{0}$ then
			\[ \rho(g) f(x) =  \xi(J) q^2 \E_{y \in (\Z/q\Z)^4} f(y) e_q(x^T y);\]
			\item \textup{(Quadratic modulations)} If $g = l(W)$ where $l(W) := \smallpmat{I}{0}{W}{I}$ with $W$ symmetric then 
			\[ \rho(g) f(x) = \xi(l(W))e_q(-\frac{1}{2} x^T W x)f(x).\]
            \item \textup{(An explicit expression)} If $q=p$ is an odd prime then 
			\[ \xi(s(E))=\Big(\frac{\det E}{p}\Big).\]
		\end{itemize}
	\end{proposition} 

Throughout this paper, the letter $\rho$ denotes the Weil representation in Proposition \ref{weil-prop}.

\begin{definition}[Symplectic element]\label{symplectic-element}
		Suppose that $Q(x,y) = x^T a x + x^T b y + y^T c y$ is a quadratic form and that $\det b \neq 0$. Then we associate to $Q$ the element $g = g(Q) \in \Sp_8(\Q)$ defined by
		\begin{equation}\label{sym}
			g :=  \pmat{-2 b^{-T} c}{b^{-T}}{4ab^{-T} c - b}{-2a b^{-T}}.
		\end{equation}  
		We call this the \emph{symplectic element} associated to $Q$.
	\end{definition}

For $h = \smallpmat{\alpha}{\beta}{\gamma}{\delta}$ and $r \in (\Z/q\Z)^*$, we define the dilate of $h$ by 
	\begin{equation}\label{gr-def}h^{(r)} = \pmat{\alpha}{r^{-1} \beta}{r\gamma}{\delta}.\end{equation} 
Note that when $(q, \det b)=1$, if we view $g\in \Sp_8(\Z/q\Z)$ then
$$g^{(r)}\in \Sp_8(\Z/q\Z) \ \textrm{for all}\ r \in (\Z/q\Z)^*.$$	
Green introduced the following group, which plays an important role in the proof.
	\begin{definition}\label{gamma-def} Suppose that $q$ is odd, squarefree and coprime to $\det b$. Let $\Gamma_q \leqslant \Sp_8(\Z/q\Z)$ be the group generated by the elements $g^{-(r)} g^{(s)}$, $r, s \in (\Z/q\Z)^*$.
	\end{definition}
	
	$\Gamma_q$ behaves nicely under projection. 
	\begin{lemma}[Proposition9.4 \cite{Green}]\label{gam-prod}
		Suppose that $q$ is squarefree, and let
        $$\pi : \Sp_8(\Z/q\Z) \rightarrow \prod_{p | q}\Sp_8(\Z/p\Z)$$
        be the natural isomorphism. Then $\pi(\Gamma_q) = \prod_{p | q} \Gamma_p$.
	\end{lemma} 

After the interpretation of exponential sums by using Weil representation, to prove Proposition \ref{proposition51}, it suffices to prove the following.
	\begin{proposition}\label{quasirandom-gamma}
		Let $Q$ be a pan-generic quadratic form. Then there is some $p_0(Q)$ such that the following is true. Suppose that $q$ is squarefree and has all prime factors greater than $p_0(Q)$. Then $\rho|_{\Gamma_q}$ splits into irreducible subrepresentations of dimensions $\geqslant q^{1 - o(1)}$. 
	\end{proposition}

The proof of Proposition \ref{proposition51} from Proposition \ref{quasirandom-gamma} can be found in \cite[Sections 7-8]{Green}. We briefly explain the proof for completeness. By \cite[Proposition 7.4]{Green}, we have
  $$|T_{f_1,f_2}(r)|=|\langle \overline{f_1}, \rho(g^{(r)}) f_2\rangle|.$$ 
  Then \eqref{zq-main-ineq} is equivalent to
  \begin{equation}\label{zq-main-ineq rewrite}
			\int  |\langle f_1, \rho(x) f_2\rangle| d\mu_q(x) \ll q^{-c_0} \Vert f_1 \Vert_2 \Vert f_2 \Vert_2,
		\end{equation} 
where $\mu_q$ is a probability measure on $\Sp_8(\Z/q\Z)$, which puts weight $\frac{1}{\phi(q)}$ on each of the points $g^{(r)}$, $r \in (\Z/q\Z)^*$. According to \cite[Proposition 8.2]{Green}, the inequality \eqref{zq-main-ineq rewrite} can be obtained from the almost uniform distribution of convolution powers and a quasirandomness result. Note that \cite[Proposition 8.4]{Green} provides the almost uniform distribution of convolution powers. We remark that \cite[Proposition 8.4]{Green} holds for all quadratic forms $Q$. However, the quasirandomness was proved in \cite{Green} under the condition that $Q$ is generic (see \cite[Proposition 8.5]{Green}). Now Proposition \ref{quasirandom-gamma} provides the quasirandomness. This proves Proposition \ref{proposition51} (subject to the truth of Proposition \ref{quasirandom-gamma}).

As in \cite{Green}, Proposition \ref{quasirandom-gamma} can be reduced to the prime modulus case. The details can be found in \cite[p. 1628]{Green}. Thus, the main task in this paper is to prove the following.
	\begin{proposition}\label{key-nontrivial}
		Suppose that $Q$ is pan-generic and that $p$ is sufficiently large in terms of $Q$. Let $\rho : \Sp_8(\Z/p\Z)  \rightarrow \operatorname{U}(\ell^2((\Z/p\Z)^4))$ be the mod $p$ Weil representation. Let $r \in (\Z/p\Z)^*$, and let $\tilde \rho$ be the twist of $\rho$ by dilation by $r$, that is to say $\tilde\rho(x) = \rho(x^{(r)})$. Then $\tilde \rho|_{\Gamma_p}$ splits into irreducible subrepresentations of dimension $\geqslant (p-1)/2$.
	\end{proposition}

\bigskip

\section{Identifying ${\Gamma}_p$ }

In this section, we study the structure of the group $\Gamma_q$.
	
Following Green, we consider the ``block DUL factorisation'' of $g$ 
	\begin{equation}\label{dul-fact} g = \pmat{A}{0}{0}{A^{-T}} \pmat{I}{B}{0}{I} \pmat{I}{0}{C}{I},\end{equation} 
where $A=-(2ab^{-T})^{-T}$ and 
	\begin{equation} \label{BC} B = -2b^{-1} a b^{-T}, \quad C = -2c + \frac{1}{2} b^T a^{-1} b.\end{equation}
It is well-defined since $\det ab\neq 0$. Moreover,
$$BC = \Delta.$$ 
	So that for $g^{-(r)}g^{(s)}$ the generators of $\Gamma_p$, we have 
	\begin{equation}\label{grs} g^{-(r)}g^{(s)} = \pmat{I}{0}{-rC}{I} \pmat{I}{-r^{-1} B}{0}{I} \pmat{I}{s^{-1} B}{0}{I} \pmat{I}{0}{sC}{I}.\end{equation}
	Now set \begin{equation}\label{tau-def}\tau := \pmat{I}{0}{0}{B} \in \GL_8(\Q).\end{equation} 
	Note that $\det(B)\not=0$, so $\tau$ is invertible over $\F_p$ for $p\nmid \det(B)$ and  may be considered as an element of $\GL_8(\Z/p\Z)$. Now observe that for $\lambda \in \F_p^*$ we have
	\begin{equation}\label{ul-conjs} \pmat{I}{\lambda B}{0}{I} = \tau^{-1} \pmat{I}{\lambda I}{0}{I} \tau, \quad \pmat{I}{0}{\lambda C}{I} = \tau^{-1} \pmat{I}{0}{\lambda \Delta}{I} \tau.\end{equation}
	Hence, we have 
	\begin{equation}
		g^{-(r)}g^{(s)}=\tau ^{-1}\pmat{I}{0}{-r \Delta}{I}\pmat{I}{-r^{-1} I}{0}{I}\pmat{I}{s^{-1} I}{0}{I} \pmat{I}{0}{s \Delta}{I} \tau.
	\end{equation}
	Therefore, the group $\Gamma_{p}$, generated by $g^{-(r)}g^{(s)}$, is a subgroup of $ \tau ^{-1}\SL_2(\F_p[\Delta])\tau$.

	\begin{lemma}\label{baby case}
		Suppose that $p \geqslant 5$ and that $\theta \in \overline{\F}_p $. We define $M_\theta$ to be the subgroup of $\SL_2(\F_p(\theta))$ generated by the matrices 
		\[ M_{\theta}(r,s) =  \pmat{1}{0}{-r\theta}{1} \pmat{1}{-r^{-1}}{0}{1}\pmat{1}{s^{-1}}{0}{1}\pmat{1}{0}{s\theta}{1} ,\] as $r,s$ range over $\F_p^*$. Then 
		
(i) For $\theta \in \F_p\setminus\{-1,0\}$, we have  $M_\theta=\SL_2(\F_p(\theta))$.

(ii) For $\theta=0$, we have 
			\begin{equation}\label{Up}M_0=U_p:=\{A=\smallpmat{1}{s}{0}{1} \,|\, s\in \mathbb{F}_{p}\}.\end{equation}

(iii) For $\theta=-1$, we have 
			\begin{equation}\label{Rp}M_{-1}=R_p:= \{A\in \mathrm{SL}_2(\mathbb{F}_{p}) \mid A \text{ is upper triangular}\}.\end{equation}
	\end{lemma}
	\begin{proof}
		It is obvious that $M_\theta$ is a subgroup of $\SL_2(\F_p(\theta))$.
	
(i) This is \cite[Proposition 10.5]{Green}.

(ii) This is obvious.

(iii) Direct computation gives, for any $r,s\in\mathbb F_p^*$, $M_{-1}(r,s)=\smallpmat{s/r}{(r-s)/rs}{0}{r/s}$, so $M_{-1}\subseteq R_p$. Then we put $a=s/r$, so $s=ar$. For $a\neq 1$, the entry $b:=(r-s)/rs=(1-a)/(ar)$ runs through all nonzero values as $r$ varies over $\mathbb F_p^*$. Hence $M_{-1}$ contains every matrix of the form $\smallpmat{a}{b}{0}{a^{-1}}$ where $a\neq 1$ and $b\neq 0$. Thus $\frac{1}{2}(p-1)^2<(p-1)(p-2)\leqslant |M_{-1}|$. Note that $|R_p|=p(p-1)$. Therefore, $M_{-1}=R_p$.\qedhere
	\end{proof}
	
Let $\overline{\rho}_\Delta$ be the characteristic polynomial of $\Delta$ mod $p$, and suppose that the factorization of $\overline{\rho}_{\Delta}$ into irreducible polynomials over $\F_p$ is $f_1\cdots f_n$. Suppose that $p$ is sufficiently large. Then $\overline{\rho}_{\Delta}$ has distinct roots in $\overline{\F}_p$, and the $f_i$ will be coprime.
	For each $i$, let $\alpha_i \in \overline{\F}_p$ be a root of $f_i$.
	
	We recall the isomorphism introduced in \cite{Green}:
	\begin{equation}\label{ring-isom}
		\Phi :  \F_p[\Delta] \rightarrow  \prod_{i=1}^n \F_p(\alpha_i), \quad \Phi(F(\Delta)) = (F(\alpha_1),\dots, F(\alpha_n))
	\end{equation}
	for any $F \in \F_p[X]$. The map $\Phi$ is a well-defined ring homomorphism.	The map $\Phi$ induces a group isomorphism
	\begin{equation}\label{group-iso}
		\Phi : \SL_2(\F_p[\Delta]) \rightarrow \prod_{i=1}^n \SL_2(\F_p(\alpha_i)).
	\end{equation}

	Let $\gamma(r,s) := \tau g^{-(r)} g^{(s)} \tau^{-1}$, then we have
	\begin{equation}\label{phi}
		\Phi(\gamma(r,s)) = (M_{\alpha_i}(r,s))_{i = 1}^n,
	\end{equation}
	where $M_{\alpha_i}(r,s)$ is defined in Lemma \ref{baby case}.

	We also need a couple of lemmas about subgroups of direct products. We shall repeatedly use the following Goursat's lemma (see \cite[Proposition 1.6]{S16}).
	
	\begin{lemma}[Goursat's lemma]\label{goursat}
		Let $G_1, G_2$ be groups. Consider the direct product $G_1 \times G_2$ and let $\pi_i : G_1 \times G_2 \rightarrow G_i$ be the two projection maps. Let $H \leqslant G_1 \times G_2$ be a subgroup, and suppose that $\pi_i(H) = G_i$ for $i = 1,2$. Then there are normal subgroups $N_i \lhd G_i$ and an isomorphism $\phi : G_1/N_1 \rightarrow G_2/N_2$ such that $H$ has the form $\{(g_1, g_2) \in G_1 \times G_2 : \phi(\overline{g}_1) = \overline{g}_2\}$, where $\overline{g}_i$ is the image of $g_i$ in $G_i/N_i$.
	\end{lemma}

	We note the following fact concerning the group $\mathrm{SL}_2$:
	\begin{lemma}\label{fact of sl2}
		Suppose that $p$ is a prime larger than $3$.
		
(i) Let $q$ be a power of $p$. Then the center of $\mathrm{SL}_2(\mathbb{F}_{q})$,  $Z\big(\mathrm{SL}_2(\mathbb{F}_{q})\big)=\{I,-I\}$,  is the unique nontrivial normal subgroup of $\mathrm{SL}_2(\mathbb{F}_{q})$.

(ii) Suppose that the group $G=\prod_{j=1}^{m}\mathrm{SL}_2(\mathbb{F}_{p^{d_j}})$ and $H$ is a proper normal subgroup of $G$. Then $p(p^2-1)/2$ divides $|G/H|$. 
	\end{lemma}	
	\begin{proof}
(i) See \cite[Chapter XIII, Theorem 9.3]{Lang}.\label{sing}

(ii) Let $S_i$ be the subgroup of $G$ consisting of elements with all $j$-th ($j\not=i$) coordinates equal to $\{I\}$. Note that $S_i\cong \mathrm{SL}_2(\mathbb{F}_{p^{d_i}})$. For notation simplicity, we may view $S_i= \mathrm{SL}_2(\mathbb{F}_{p^{d_i}})$. Since $H$ is proper, there exists $i$ such that 
$S_i$ is not contained in $H$. Then $K:=S_i\cap H$ is a proper normal subgroup of $S_i$. By (i), $K\subseteq \{I,-I\}$. Consider the group homomorphism 
$\varphi: S_i\rightarrow G/H$ defined by $\varphi(s)=sH$, the kernel of $\varphi$ is exact $K$.  Applying the first isomorphism theorem and Lagrange's theorem, we conclude that $p^{d_i}(p^{2d_i}-1)/|K|=|S_i/K|$ divides $|G/H|$. Thus, $p(p^2-1)/2$ divides $|G/H|$.  
	\end{proof}

If the group $G$ is in the form $G=\prod_{j=1}^{m}\mathrm{SL}_2(\mathbb{F}_{p^{d_j}})$, then we say $G$ is a finite direct product of $\mathrm{SL}_2(\mathbb{F}_{p^{d}})$.
	\begin{lemma}\label{Ribet2}
		Suppose that both $G_1$ and $G_2$ are finite direct products of $\mathrm{SL}_2(\mathbb{F}_{p^d})$. Let $H\leqslant G_1$ with $|H|$ dividing $p^2(p-1)$. Let $\pi_{i}:G_1\times G_2\to G_i$ be projection for $i=1,2$. Suppose that $G\leqslant G_1\times G_2$ satisfies $\pi_{1}(G)=H$ and $\pi_{2}(G)=G_2$. Then $G=H \times G_2$.
	\end{lemma}
	
	\begin{proof}
		Clearly, $G\leqslant H\times G_2$. By Lemma \ref{goursat}, there exist normal subgroups $N_1\triangleleft H$, $N_2\triangleleft G_2$ and isomorphism $\varphi: H/N_1\to  G_2/N_2$ such that
		\begin{equation*}
			G=\{(x,y)\in  H\times  G_2:\varphi(\overline{x})=\overline{y}\},
		\end{equation*}
		where $\overline{x}$ and $\overline{y}$ are the image of $x$ in  $  H/N_1$ and $y$ in $ G_2/N_2$. It suffices to prove that $N_1=H$ (or equivalently $N_2=G_2$). 
Otherwise, $N_2$ is a proper normal subgroup of $G_2$. Then by Lemma \ref{fact of sl2} (ii), $p(p^2-1)/2$ divides $|G_2/N_2|$. Since $|G_2/N_2|=|H/N_1|$, we further obtain  $p(p^2-1)/2\mid |H|$. This is a contradiction since $|H|$ divides $p^2(p-1)$.
	\end{proof}
Note that $|U_p|=p$ and $|R_p|=p(p-1)$, 
		Lemma \ref{Ribet2} holds for $H=U_p$, $R_p$ and $U_p\times R_p$.

\begin{lemma}\label{Ribet3} Let $p>3$ be a prime. Let $H_1=U_p$ and $H_2=R_p$.
		Let $p_{i}:H_1\times H_2\to H_i$ be projection for $i=1,2$. Suppose that $H\leqslant H_1\times H_2$ satisfies $p_{1}(H)=H_1$ and $p_{2}(H)=H_2$. Then $H=H_1 \times H_2$.
	\end{lemma}
	
	\begin{proof}By Lemma \ref{goursat}, there are normal subgroups $N_1 \lhd U_p$ and $N_2\lhd R_p$ and an isomorphism $\phi : U_p/N_1 \rightarrow R_p/N_2$ such that $H$ has the form $\{(g_1, g_2) \in U_p \times R_p : \phi(\overline{g}_1) = \overline{g}_2\}$. It suffices to prove $N_1=U_p$. Otherwise, $N_1=\{I\}$ and then $|N_2|=p-1$. Since $U_p$ is a subgroup of $R_p$, we have $N_2\cap U_p=\{I\}$. Take an arbitrary $h\in N_2$. Write $h=\pmat{a}{b}{0}{a^{-1}}$. Set $u=\pmat{1}{1}{0}{1}\in U_p$. A direct calculation gives $huh^{-1}u^{-1}=\pmat{1}{a^2-1}{0}{1}$. Note that $huh^{-1}u^{-1}\in N_2\cap U_p$. Thus, $a=1$ or $a=-1$. 

If $a=1$, then $h\in N_2\cap U_p$ and thus $h=I$. If $a=-1$, then $h^2=\pmat{1}{-2b}{0}{1}\in N_2\cap U_p$ and thus $h^2=I$ which implies $b=0$ and $h=-I$. Note that $h$ is arbitrary in $N_2$, and this is a contradiction to $|N_2|=p-1$. This completes the proof.   
	\end{proof}

Recall that $\alpha_i$ is the root of $f_i$, where $\prod_{i=1}^nf_i$ is the characteristic polynomial of $\Delta$ mod $p$. Write $\mathcal{R}=\{\alpha_1,\ldots,\alpha_n\}$. When $\mathcal{R}\cap \{0,-1\}=\emptyset$, Proposition \ref{quasirandom-gamma} holds by \cite[Proposition 8.5]{Green}. 
We only need to consider the case $\mathcal{R}\cap \{0,-1\}\not=\emptyset$. We may assume that $\alpha_1=0$ if $\mathcal{R}\cap \{0,-1\}=\{0\}$, $\alpha_1=-1$ if $\mathcal{R}\cap \{0,-1\}=\{-1\}$ and $\alpha_1=0,\alpha_2=-1$ if $\{0,-1\}\subseteq \mathcal{R}$. Then we write $\mathcal{I}=\{1\}$ or $\{1,2\}$ according to $|\mathcal{R}\cap \{0,-1\}|=1$ or $2$. Let 
\begin{equation}\label{definecalJ}
 \mathcal{J}=\{1,2,\ldots,n\}\setminus \mathcal{I}.
\end{equation}

Note that for $j\in \mathcal{J}$, $M_{\alpha_j}=\SL_2(\F_p(\alpha_j))$. 
\begin{lemma}\label{Ribet1} With the notations $\mathcal{I}$ and $\mathcal{J}$ as the above. Let $H=\prod_{i\in \mathcal{I}}M_{\alpha_i}$ and $G_2=\prod_{j\in \mathcal{J}}M_{\alpha_j}$. Write $G=\Phi(\tau\Gamma_p\tau^{-1})$. Let $\pi_2:H\times G_2\to G_2$ be the projection. Then $\pi_{2}(G)=G_2$.
	\end{lemma}
	\begin{proof}This is contained in the proof of \cite[Proposition 10.1]{Green}.
	\end{proof} 

	We introduce the structure of $\Gamma_p$.
	\begin{lemma}\label{gamma_p}
		Suppose that $Q$ is pan-generic and that $p$ is sufficiently large in terms of $Q$. Then we have 
		\begin{equation*}
			\Gamma_{p} \leqslant \tau ^{-1}\SL_2(\F_p[\Delta]) \tau \quad  \text{and} \quad  \Phi(\tau\Gamma_p \tau^{-1})=\prod_{i=1}^n M_{\alpha_i},
		\end{equation*}
		where  $\Phi$ is the isomorphism defined in \eqref{group-iso}.
	\end{lemma}
\begin{proof} This is a generalization of \cite[Proposition 10.1]{Green}. Let $H=\prod_{i\in \mathcal{I}}M_{\alpha_i}$ and $G_2=\prod_{j\in \mathcal{J}}M_{\alpha_j}$. Write $G=\Phi(\tau\Gamma_p\tau^{-1})$. The task is to prove that $G=\prod_{i=1}^n M_{\alpha_i}$. Let $\pi_1:H\times G_2\to H$ and $\pi_2:H\times G_2\to G_2$ be projections.

By Lemma \ref{Ribet1}, we have $\pi_{2}(G)=G_2$. In the case $|I|=1$, Lemma \ref{baby case} gives $\pi_1(G)=H$. In the case $|I|=2$, we combine Lemma \ref{baby case} and Lemma \ref{Ribet3} to obtain $\pi_1(G)=H$. Then finally, we deduce from Lemma \ref{Ribet2} that $\Phi(\tau\Gamma_p\tau^{-1})=G=H\times G_2=\prod_{i=1}^n M_{\alpha_i}$. This completes the proof of Lemma \ref{gamma_p}.

\end{proof}

\section{Proof of Proposition \ref{quasirandom-gamma}}

Throughout this section, we assume that $Q$ is pan-generic, $p$ is sufficiently large in terms of $Q$ and $\overline{\rho}_{\Delta}$ has distinct roots in $\overline{\F}_{p}$. In particular, the irreducible factors $f_{i}(x)\in \F_p[x]$ are pairwise coprime. This yields the direct sum decomposition of subspaces:
	\begin{equation}\label{deco-space}
		V:=\F_p^4=\oplus_{i=1}^{n}W_i,\quad\ W_i=\ker f_i(\Delta).
	\end{equation}
We use $(\mathbb{F}_{p}[\Delta])^{\times}$ to denote the set of invertible elements in $\mathbb{F}_{p}[\Delta]$.
	
		\begin{lemma}\label{x}
		The set
		\begin{equation*}
			X_p:=\{F(\Delta)\in (\mathbb{F}_{p}[\Delta])^{\times}: F(x)\in \mathbb{F}_{p}[x], F(0)=F(-1)=1\}
		\end{equation*}
		is a multiplicative subgroup of $(\mathbb{F}_p [\Delta])^\times$. 
	\end{lemma}
	\begin{proof}
		Obviously, $X_p$ is closed under multiplication. Note that when $F(\Delta)\in X_p$, $F$ must be coprime with $\overline{\rho}_\Delta(x)x(x+1)$. By B\'ezout's theorem, there exist $a(x),b(x)\in \mathbb{F}_{p}[x]$ such that
		\begin{equation}\label{bezout}
			a(x)F(x)+b(x)\overline{\rho}_\Delta(x)x(x+1)=1.
		\end{equation}
Since $\overline{\rho}_\Delta(\Delta)=0$, by (\ref{bezout}) we have $a(\Delta)F(\Delta)=I$. Taking $x=0,-1$ in (\ref{bezout}), we have $a(0)=a(-1)=1$. Thus, $F(\Delta)^{-1}=a(\Delta)\in X_p$. This completes the proof of this lemma.
	\end{proof}

Recall $\mathcal{I}$ and $\mathcal{J}$ defined around \eqref{definecalJ}. We introduce 
		\begin{equation}\label{defineTheta}
			\Theta:=\Phi^{-1}\bigg(\prod_{i \in \mathcal{I}}\{I\} \times \prod_{i \in \mathcal{J}}\SL_2(\F_p(\alpha_i))\bigg)
		\end{equation}
and \begin{equation}\label{defineThetap}
			\Theta_p:=\tau^{-1}\Theta \tau,
		\end{equation}
		where $\Phi$ is the group isomorphism in \eqref{group-iso}. Then by Lemma \ref{gamma_p}, $\Theta$ is a subgroup of $\Gamma:=\tau\Gamma_p\tau^{-1}$ and $\Theta_p$ is a subgroup of $\Gamma_p$. 

We define
		\begin{equation}\label{y}
			Y_p:=\{F(\Delta)\in \mathbb{F}_{p}[\Delta]: F(x)\in \mathbb{F}_{p}[x],F(0)=F(-1)=0\}.
		\end{equation}
Recall notations in Proposition \ref{weil-prop}. We shall consider
	\begin{equation}
		s(P)=\begin{pmatrix}
			P& 0\\0& P^{-T}
		\end{pmatrix},\quad
		l(CN)=\begin{pmatrix}
			I& 0\\CN& I
		\end{pmatrix},
	\end{equation}
	where $P,N\in \F_p[\Delta]$, and $C$ is given in \eqref{BC}.

	\begin{lemma}\label{usl}
		Let $P\in X_p$ and $N\in Y_p$. Then we have
		 $s(P), l(CN)\in \Theta_p$.
	\end{lemma}
	\begin{proof}
Recall that $\Delta=BC$, where $B$ and $C$ are symmetric. We conclude that $\Delta^{m} B=B(\Delta^T)^{m}$ for any $m\in \mathbb{N}$,  and $F(\Delta)B=BF(\Delta)^{T}$ for any polynomial $F(x)\in \F_p[x]$. Hence,
$PB=BP^{T}$ and then $BP^{-T}B^{-1}=P^{-1}$. Then
		\begin{equation}\label{conjsP}
			\tau s(P) \tau^{-1}=
			\begin{pmatrix}
				P &0\\ 0& BP^{-T}B^{-1}
			\end{pmatrix}=
			\begin{pmatrix}
				P &0\\ 0& P^{-1}
			\end{pmatrix}.
		\end{equation}
We also have
		\begin{equation}\label{conjlCN}\tau  
			l(CN)\tau^{-1}=
			\begin{pmatrix}
				I & 0\\ \Delta N &I
			\end{pmatrix}.
		\end{equation}
Thus $\tau s(P) \tau^{-1},\tau  
			l(CN)\tau^{-1}\in \SL_2(\F_p[\Delta])$.

Note that $s(P), l(CN)\in \Theta_p$ is equivalent to $\tau s(P) \tau^{-1},\tau  
			l(CN)\tau^{-1}\in \Theta$.
Recall $\Phi$ in \eqref{group-iso}, and let $\pi_i$ be the projection of $\prod_{i=1}^n \SL_2(\F_p(\alpha_i))$.  It suffices to prove for each $i\in \mathcal{I}$ that
\begin{equation}\label{maptoI}
  \pi_i\circ\Phi(\tau s(P) \tau^{-1})=\pi_i\circ\Phi(\tau  
			l(CN)\tau^{-1})=I.
\end{equation}

Write $P=F_1(\Delta)$ with $F_1(0)=F_1(-1)=1$ and $N=F_2(\Delta)$ with $F_2(0)=F_2(-1)=0$. Then by \eqref{ring-isom}, \eqref{conjsP} and \eqref{conjlCN}, we have
		\begin{equation*}
			\pi_i\circ\Phi(\tau s(P) \tau^{-1})=\pmat{F_1(\alpha_i)}{0}{0}{F_1(\alpha_i)^{-1}},\ \pi_i\circ\Phi(\tau  
			l(CN)\tau^{-1})=\pmat{1}{0}{\alpha_iF_2(\alpha_i)}{1}.
		\end{equation*}
Therefore, \eqref{maptoI} holds for $i\in \mathcal{I}$. This completes the proof of this lemma.
	\end{proof}

	\begin{lemma}[Proposition 5.4 in \cite{Helfgott}]\label{sl2-facts}
		Let $k$ be a finite field of odd characteristic. Then the smallest nontrivial complex representation of $\SL_2(k)$ has dimension at least $\frac{1}{2}(|k| - 1)$. 
	\end{lemma}

Now we explain that Proposition \ref{key-nontrivial} can be deduced from the following.
\begin{proposition}\label{key-nontrivial 2}
		Suppose that $Q$ is pan-generic and that $p$ is sufficiently large in terms of $Q$. Let $\rho : \Sp_8(\Z/p\Z)  \rightarrow \operatorname{U}(\ell^2((\Z/p\Z)^4))$ be the mod $p$ Weil representation. Let $r \in (\Z/p\Z)^*$, and let $\tilde \rho$ be the twist of $\rho$ by dilation by $r$, that is to say $\tilde\rho(x) = \rho(x^{(r)})$. Then $\tilde \rho|_{\Theta_p}$ has no nontrivial invariant vector.
	\end{proposition}
	\begin{proof}[Proof of Proposition \ref{key-nontrivial} assuming Proposition \ref{key-nontrivial 2}]
		Let $\eta:\Gamma_p\to GL(W_0)$ be an irreducible subrepresentation of $\tilde{\rho}|_{\Gamma_p}$, where $W_0$ is the invariant subspace. If $\eta|_{\Theta_p}$ is
		trivial, then every nonzero vector in $W_0$ would be invariant under $\tilde \rho|_{\Theta_p}$, contradicting to Proposition \ref{key-nontrivial 2}. Hence $\eta|_{\Theta_p}$ is non-trivial. Recall \eqref{defineTheta}, then $\Theta_p$ is isomorphic to $\prod_{i \in \mathcal{I}}\{I\} \times \prod_{i \in \mathcal{J}}\SL_2(\F_p(\alpha_i))$. Considering the tensor product, we conclude from 
		Lemma \ref{sl2-facts} that the irreducible subrepresentation of $\eta|_{\Theta_p}$ has dimension at least $(p-1)/2$. Therefore we have 
			$\dim \eta \geqslant  (p-1)/2$,
		which implies Proposition \ref{key-nontrivial}.
	\end{proof}

	\begin{lemma}\label{quad-non}
		Suppose $p$ is a prime larger than $3$. Then there exists $P_0\in X_p$ such that $\det P_0$ is a quadratic non-residue modulo $p$.
	\end{lemma}
	\begin{proof}
	Recall the notations $W_i$ in \eqref{deco-space}.  Recall \eqref{definecalJ}, and
		fix an index $j\in \mathcal{J}$. The main task is to prove that there exists $g(x)\in \F_p[x]$ such that $\det(g(\Delta)|_{W_j})$ is a quadratic non-residue modulo $p$. Suppose that $f_j(x)=\prod_{k=1}^{d}(x-\beta_k)$, where $d=\deg(f_j)$ and we may assume that $\beta_1=\alpha_j$. Then
$$\det(g(\Delta)|_{W_j})=\prod_{k=1}^dg(\beta_k)=N_{L/\F_p}(g(\alpha_j)),$$
where $L=\F_p(\alpha_j)$. 

Let $a_0\in \F_p$ be a quadratic non-residue. Since $N_{L/\F_p}(\cdot)$ is surjective, there exists $\gamma\in L$ such that $N_{L/\F_p}(\gamma)=a_0$. 
Write $\gamma=\sum_{t=0}^{d-1}c_t\alpha_j^t$ with all $c_t\in \F_p$. Taking $g(x)=\sum_{t=0}^{d-1}c_tx^t\in \F_p[x]$, we have $g(\alpha_j)=\gamma$ and thus $\det(g(\Delta)|_{W_j})=a_0$. Now choose $F(x)$ such that 
$F(x)\equiv g(x)\pmod{f_j(x)}$, $F(x)\equiv 1\pmod{x(x+1)}$ and $F(x)\equiv 1\pmod{f_i(x)}$ for all $i\not=j$.  
Then $\det(F(\Delta))=\prod_{i}\det(F(\Delta)|_{W_i})=\det(g(\Delta)|_{W_j})=a_0$. We complete the proof by choosing $P_0=F(\Delta)$. 
	\end{proof}
	
	\begin{lemma}\label{invariant}
		 $\Theta_p$ is invariant under the dilation defined in  \eqref{gr-def}.
	\end{lemma}
	\begin{proof}
		For each $r\in \mathbb{F}_p^*$ and any ring $R$ over $\mathbb{F}_p$, we may consider the dilation operator $D_r$ on $\SL_{2}(R)$ with the same definition as \eqref{gr-def}. For convenience, we shall use the operator notation $D_r$ without distinguishing the ring on which it acts. 
Note that $D_{r}$ is an automorphism of $ \SL_2(R)$. Thus it suffices to prove that $D_r(\Theta_p)\subseteq \Theta_p$.

A direct verification reveals that 
	\begin{equation*}\pmat{I}{0}{0}{B^{-1}}\pmat{\alpha}{r^{-1} \beta}{r\gamma}{\delta}\pmat{I}{0}{0}{B}
= \pmat{\alpha}{r^{-1} \beta B}{rB^{-1}\gamma}{B^{-1}\delta B}.\end{equation*} 
In particular, $D_{r}$ commutes with the conjugation $x\mapsto\tau^{-1} x\tau (x\in \SL_{2}(R))$. Thus, $D_r(\Theta_p)\subseteq \Theta_p$ is equivalent to 
$D_r(\Theta)\subseteq \Theta$. For $M=\pmat{F_{11}}{F_{12}}{F_{21}}{F_{22}}$ with $F_{ij}(x)\in \F_p[x]$ and $\alpha\in \overline{\F}_p$, we use the notations
$$M(\Delta)=\pmat{F_{11}(\Delta)}{F_{12}(\Delta)}{F_{21}(\Delta)}{F_{22}(\Delta)}\quad \textrm{ and }\quad 
M(\alpha)=\pmat{F_{11}(\alpha)}{F_{12}(\alpha)}{F_{21}(\alpha)}{F_{22}(\alpha)}.$$
Recall the isomorphisms in \eqref{ring-isom} and \eqref{group-iso}, we conclude that 
		for $M(\Delta)\in \SL_2(\F_p[\Delta])$,
		\begin{equation*}
			\Phi(D_r(M(\Delta)))=\big(D_r(M(\alpha_i)) \big)_{i=1}^n.
		\end{equation*} 
For $M(\Delta)\in \Theta$ and $i\in \mathcal{I}$, we have $M(\alpha_i)=I$ and then $D_r(M(\alpha_i))=I$. Thus, $D_r(M(\Delta))\in \Theta$. This proves $D_r(\Theta)\subseteq \Theta$, and consequently $D_r(\Theta_p)\subseteq \Theta_p$. This completes the proof of this lemma.
	\end{proof}

\begin{proof}[Proof of Proposition \ref{key-nontrivial 2}] 
Due to Lemma \ref{invariant}, it suffices to consider the case $\tilde\rho|_{\Theta_p} = \rho|_{\Theta_p}$. Suppose
		\begin{equation}\label{invariance-cond}  \rho(g) f = f\end{equation} 
		for some $f:V\to\mathbb{C}$ and any $g \in \Theta_p$. 
		Our aim is to show that $f$ is identically zero.

		We formulate the action of $l(CN)$ from Proposition \ref{weil-prop}, namely
		\begin{equation*}
			\rho(l(CN)) f(x) =\xi(l(CN)) e_p(-\frac{1}{2} x^T CN x) f(x)
		\end{equation*}
		for some unit complex number $\xi(l(CN))$ independent of $x$. We shall choose $N\in Y_p$. By Lemma \ref{usl} we know that $l(CN)\in \Theta_p$, and by \eqref{invariance-cond},  
		\begin{equation*} \rho(l(CN)) f(x) = f(x).\end{equation*} 
It follows from above that there exists $t_N\in \mathbb{F}_{p}$ such that 
		\begin{equation}\label{supp f}
			\Supp({f})\subseteq \{x\in \mathbb{F}_{p}^4:x^TCNx=t_N\}.
		\end{equation}

\medskip

\noindent {\bf Claim 1.} $\Supp({f})\subseteq \oplus_{i\in \mathcal{I}}W_i$. 

Otherwise we assume that $v\in \Supp({f})$ but $v\not\in \oplus_{i\in \mathcal{I}}W_i$. Thus we can write $v$ in the form $v=\sum_{i}w_i$ with $w_i\in W_i$ for all $i$ and $w_j\not=0$ for some $j\in \mathcal{J}$. Now we fix this $j$. Note that $g(\Delta)(w_j)$ are pairwise distinct for $g(x)\in \F_p[x]/(f_j(x))$ since $f_j$ is minimal. Thus, by counting the number of elements we get that $g(\Delta)(w_j)$ runs over $W_j$ as $g(x)\in \F_p[x]/(f_j(x))$.

Now choose any $w\in W_j\setminus\{0\}$ and $\lambda\in \F_p^\ast$. From above, there exists $h(x)$ such that $h(\Delta)(w_j)=\lambda w$. Then there exists $F(x)\in \F_p[x]$ such that $F(x)\equiv h(x)\pmod{f_j(x)}$, $F(x)\equiv 1\pmod{x(x+1)}$ and $F(x)\equiv 1\pmod{f_i(x)}$ for all $i\not=j$. Then $F(\Delta)\in X_p$ and write $P=F(\Delta)$. By Lemma \ref{usl}, $s(P)\in \Theta_p$.  Moreover,
$$Pv=F(\Delta)v=\sum_{i}F(\Delta)(w_i)=\lambda w+\sum_{i\not=j}w_i.$$

By the dilation property in Proposition \ref{weil-prop}, there exists a unit complex number $\xi(s(P))$ such that
		\begin{equation*}
			\rho(s(P))f(x)=\xi(s(P))f(P^{-1}x).
		\end{equation*}
Combining this with \eqref{invariance-cond}, we obtain
		\begin{equation}\label{-1}
			\xi(s(P))f(P^{-1}x)=f(x).
		\end{equation}
Inserting $x=Pv$, we obtain $f(\lambda w+\sum_{i\not=j}w_i)\not=0$, that is $\lambda w+\sum_{i\not=j}w_i\in \Supp({f})$. 

We take $N=\Delta(\Delta+I)$ in \eqref{supp f}, so there exists $t_0\in \F_p$ such that 
		\begin{equation*}
			(\lambda w+\sum_{i\not=j}w_i)^{{T}}C\Delta(\Delta+I)(\lambda w+\sum_{i\not=j}w_i)=t_0.
		\end{equation*}
The left hand side is a quadratic polynomial in $\lambda$ and $p>3$, thus 
\begin{equation}\label{forw}
			 w^{{T}}C\Delta(\Delta+I) w=0.
		\end{equation}
Then we conclude that \eqref{forw} holds for all $w\in W_j$. Note that $u_1^{{T}}C\Delta(\Delta+I) u_2=u_2^{{T}}C\Delta(\Delta+I) u_1$ since $C\Delta(\Delta+I)$ is symmetric. For each  $u_1,u_2\in W_j$, applying \eqref{forw} with $w=u_1$, $u_2$, $u_1+u_2$ respectively,  we obtain
		\begin{equation}\label{cc}
			u_1^{T}C\Delta(\Delta+I)u_2=0.
		\end{equation} 
		Consider the bilinear form: $\phi(x,y)=x^TC\Delta(\Delta+I)y$ on $\overline{\F}_p^4\times \overline{\F}_p^4$.
Then \eqref{cc} can be restated as 		
\begin{equation}\label{jtoj}
			\phi(W_j,W_j)=0.
		\end{equation} 
		
We note that $\Delta$ is self-adjoint with respect to $\phi$ since
		\begin{equation*}
			\phi(\Delta x, y) = x^T \Delta^T C\Delta(\Delta+I) y = x^T C\Delta(\Delta+I)\Delta y= \phi(x,  \Delta y).
		\end{equation*}
Let 
$$V_\alpha=\{x\in \overline{\F}_p^4:\ \Delta x=\alpha x\}.$$
Then $\phi(V_\alpha,V_\beta)=0$ for distinct eigenvalues $\alpha$ and $\beta$ of $\Delta$. The reason is that for $v_\alpha\in V_\alpha$ and $v_\beta\in V_\beta$, 
		\begin{equation*}
			\alpha \phi(v_\alpha,v_\beta)=\phi(\Delta v_\alpha,v_\beta)=\phi(v_\alpha,\Delta  v_\beta)=\beta \phi(v_\alpha,v_\beta),
		\end{equation*} 
		implying $\phi(v_\alpha,v_\beta)=0$. Note that 
$$W_i\subseteq \bigoplus_{\substack{\alpha \\ f_i(\alpha)=0}} V_{\alpha}.$$
Then we conclude that $\phi(W_i,W_{i'})=0$ for $i\not=i'$. Combining this with \eqref{jtoj}, we have
		\begin{equation}\label{jtoall}
			\phi(V,W_j)=0,
		\end{equation} 
which implies $\Delta^2(\Delta+I)(W_j)=0$. This is a contradiction since $(f_j(x),x^2(x+1))=1$. This confirms Claim 1.

\medskip

\noindent {\bf Claim 2}. $\oplus_{i\in \mathcal{I}}W_i\subseteq f^{-1}(0)$.

By Lemma \ref{quad-non}, we choose $P_0\in X_p$ and $\det P_0$ is quadratic non-residue mod $p$. Then $s(P_0)\in \Theta_p$ by Lemma \ref{usl}.  By the action of dilation in Proposition \ref{weil-prop}, for any $x\in V$ we have 
		
		\begin{equation}\label{di}
			\rho(s(P_0))f(x)=\Big(\frac{\det P_0}{p}\Big)\, f(P_0^{-1}x)=-f(P_0^{-1}x).
		\end{equation}
		Since $P_0$ acts trivially on $\oplus_{i\in \mathcal{I}}W_i$, for every
		$x\in \oplus_{i\in \mathcal{I}}W_i$ we have $P_0^{-1}x=x$. Hence it follows from \eqref{di}  that for each $x\in \oplus_{i\in \mathcal{I}}W_i$,
		\[
		\rho(s(P_0))f(x)=-f(x).
		\]
Then by the invariance condition \eqref{invariance-cond}, we have
		$\rho(s(P_0))f(x)=f(x)$. Thus $f(x)\equiv 0$ for all $x\in \oplus_{i\in \mathcal{I}}W_i$. This completes the proof of Claim 2.

Note that Claim 1 and Claim 2 together imply $f\equiv 0$. This completes the proof of Proposition \ref{key-nontrivial 2}.

\end{proof}

\section{Proof of Theorem \ref{Main}}
We set
$$c_2=24.8(2+\frac{10000}{c_0}),$$
where $c_0$ is the constant in Proposition \ref{proposition51}.
\begin{proposition}
  \label{lemmadensityincrease}Suppose that $Q$ is a pan-generic translation invariant indefinite quadratic form. Suppose that there are no distinct primes $p_1,\ldots,p_8\in \mathcal{A}\cap I$ such that $Q(p_1,\ldots,p_8)=0$. Let $\delta=\delta_{b,W;I}(\mathcal{A})$. Suppose that $\delta\geqslant (\log\log X)^{-1}$.
 Then there exist $c\in \mathbb{Z}$, $q\in \mathbb{Z}^{+}$, a subinterval $J\subseteq I$ and a positive number $\alpha_0=\alpha_0(Q)$ such that
 \begin{align*}(c,Wq)=1,\ q\leqslant \alpha_0^{-1}\delta^{-c_2}, \ \ |J|\geqslant \alpha_0\delta^{1+c_2}|I|\ \ \textrm{ and }\ \ \delta_{c,Wq;J}(\mathcal{A})\geqslant \delta(1+\alpha_0\delta^{c_2}).
 \end{align*}
\end{proposition}
 Theorem \ref{Main} follows immediately from Proposition \ref{lemmadensityincrease} with the same iterated argument at the end of \cite[Section 6]{Zhao2022}, and indeed one can take $c=(1+c_2)^{-1}$ in Theorem \ref{Main}. Hence we shall focus on the proof of Proposition \ref{lemmadensityincrease}.
For $\mathcal{A}\subseteq \mathcal{P}\cap [X/2,\,X]$, we define \begin{align*}
\nu_Q^\ast (\mathcal{A})=\sum_{\substack{x_1,\ldots,x_8\in \mathcal{A}\cap I\\ Q(\mathbf{x})=0 \\ \mathbf{x}\equiv b\pmod{W} }}\Lambda(x_1)\cdots \Lambda(x_8).\end{align*}
Let $\mathcal{A}'=\{y:\ Wy+b\in \mathcal{A}\cap I\}$. Then we can represent $\nu_Q^\ast(\mathcal{A})$ in the form
\begin{align*}
\nu_Q^\ast(\mathcal{A})=
\sum_{\substack{\mathbf{x}\\ Q(\mathbf{x})=0  }}\prod_{j=1}^8\Big(1_{\mathcal{A}'}\cdot\Lambda_{b,W;I}(x_j)\Big).
\end{align*}

For $(c,t)=1$ and $\mathcal{B}\subseteq \mathbb{N}$, we define
\begin{align*}\mathcal{N}_{c,t;I}(\mathcal{B})=
\sum_{\substack{x\equiv c\pmod{t}\\ x\in \mathcal{B}\cap I}}\Lambda(x).\end{align*}
Now we introduce the (relative) density
\begin{align}\label{definerelativedensity}\delta_{c,t;I}(\mathcal{B})
=\frac{\mathcal{N}_{c,t;I}(\mathcal{B})}{\mathcal{N}_{c,t;I}(\mathbb{N})}.\end{align}
Choosing $\delta=\delta_{b,W;I}(\mathcal{A})$,  we define $\varpi=1_{\mathcal{A}'}\cdot \Lambda_{b,W;I}-\delta\Lambda_{b,W;I}$. 
Note that \begin{align*}
\nu_Q(b,W;I)=
\sum_{\substack{\mathbf{x}\\ Q(\mathbf{x})=0  }}\prod_{j=1}^8\Lambda_{b,W;I}(x_j).
\end{align*}
Then on replacing $1_{\mathcal{A}'}\cdot \Lambda_{b,W;I}$ by $\varpi+\delta\Lambda_{b,W;I}$, we can represent $\nu_Q^\ast(\mathcal{A})-\delta^8\nu_Q(b,W;I)$ in the form
\begin{align}\label{representnu}
\nu_Q^\ast(\mathcal{A})-\delta^8\nu_Q(b,W;I)=\sum_{\varpi_1,\ldots,\varpi_8}\nu_Q^\ast(\varpi_1,\ldots,\varpi_8),
\end{align}
where we define 
\begin{align*}\nu_Q^\ast(\varpi_1,\ldots,\varpi_8)=
\sum_{\substack{\mathbf{x}\\ Q(\mathbf{x})=0  }}\prod_{j=1}^8\varpi_j(x_j)
\end{align*}
and the summation in \eqref{representnu} is taken over
\begin{align}\label{conditionvarpi}\varpi_1,\ldots,\varpi_8\in \{\varpi,\ \delta\Lambda_{b,W;I}\} \ \textrm{ with }\ \ \varpi_j=\varpi \ \textrm{
for some }\ \ 1\leqslant j\leqslant 8.\end{align}
For $\boldsymbol{\varpi}=(\varpi_1,\ldots,\varpi_8)$ satisfying \eqref{conditionvarpi}, recalling \eqref{defineStheta}, we have
\begin{align}\label{circle2}\nu_Q^\ast(\boldsymbol{\varpi})=\int_0^1S(\theta;\boldsymbol{\varpi})d\theta.
\end{align}

The definition of the major arcs is slightly different from the classical form, which we introduce in the following
\begin{align*}\mathcal{M}(M_0)=\bigcup_{1\leqslant q\leqslant
M_0}\bigcup_{\substack{r=1
 \\ (r,q)=1}}^q\mathcal{M}(q,r;M_0),\end{align*}
where the intervals $\mathcal{M}(q,r;M_0)$ are
\begin{align*}\mathcal{M}(q,r;M_0)=\Big\{\theta:\ \big|\theta-\frac{r}{q}\big|\leqslant \frac{M_0}{q(Y/W)^2}\Big\}.\end{align*}
The intervals $\mathcal{M}(q,r;M_0)$ are pairwise disjoint for $1\leqslant
r\leqslant q\leqslant M_0$ and $(r,q)=1$ provided that $2M_0<Y/W$.

\begin{lemma}[Lemma 6.1 \cite{Zhao2022}]\label{lemmadensityincmajor}Suppose that $Q$ is translation invariant and indefinite. Suppose that $\boldsymbol{\varpi}=(\varpi_1,\ldots,\varpi_8)$ satisfies \eqref{conditionvarpi}. Let $\delta=\delta_{b,W;I}(\mathcal{A})$, and let $\delta\geqslant (\log X)^{-50}$.  Let $M_0\leqslant (\log X)^{50}$.
Suppose that
\begin{align}\label{lowermajor}\Big|\int_{\mathcal{M}(M_0)}S(\theta;\boldsymbol{\varpi})d\theta\Big|\gg \frac{\delta^8 W^2Y^{6}}{\phi(W)^{8}}.\end{align}
Then there exist $c\in \mathbb{Z}$, $q\in \mathbb{Z}^{+}$, a subinterval $J\subseteq I$ and $C_Q>0$ such that
 \begin{align*}(c,Wq)=1,\ q\leqslant M_0, \ \ |J|\geqslant C_Q\delta M_0^{-3.1}|I|\ \ \textrm{ and }\ \ \delta_{c,Wq;J}(\mathcal{A})\geqslant \delta(1+C_QM_0^{-3.1}).
 \end{align*}
\end{lemma}
Lemma \ref{lemmadensityincmajor} holds for any $s$, where $s$ is the number of variables in the quadratic form $Q$. However, Lemma \ref{lemmadensityincmajor} itself cannot be used to deduce a Roth's type result unless one is equipped with desired minor arcs estimate.  If we follow the proof of \cite[Lemma 6.1]{Zhao2022}, we can see that the similar conclusion holds with $\mathcal{M}$ replaced by $\mathfrak{M}$. We plan to adopt a cheaper way to transfer $\mathcal{M}$ to $\mathfrak{M}$ in the proof of Proposition \ref{lemmadensityincrease}.

\vskip0.5cm

\noindent {\it Proof of Proposition \ref{lemmadensityincrease}.} Since the number of solutions to $Q(\mathbf{x})=0$ with $x_i=x_j$ for some $i\not=j$ is $O(X^{5+\varepsilon})$, we have $\nu_Q^\ast (\mathcal{A})\ll X^{5+\varepsilon}$.
Corollary \ref{extensioncorollary} yields
\begin{align*}\nu_Q(b,W;I)\gg \frac{W^2Y^{6}}{\phi(W)^8}.\end{align*}
Then we have
\begin{align*}|\nu_Q^\ast(\mathcal{A})-\delta^8\nu_Q(b,W;I)|\gg \frac{\delta^8W^2Y^{6}}{\phi(W)^8}.\end{align*}
In view of \eqref{representnu}, one has
\begin{align}\label{nulower}|\nu_Q^\ast(\boldsymbol{\varpi})|\gg \frac{\delta^8 W^2 Y^{6}}{\phi(W)^8}\end{align}
for some $\boldsymbol{\varpi}$ satisfying \eqref{conditionvarpi}. Let $M_0=MM'^{1+c_1}$. Recalling \eqref{defineM'} and \eqref{definec1}, we have
$M_0=M^{2+\frac{10000}{c_0}}$. For $q\leqslant M'$, one has
$\frac{Mq^{c_1}}{(Y/W)^2}\leqslant \frac{MM'^{1+c_1}}{q(Y/W)^2}$. Thus, $\mathfrak{M}(M')\subseteq \mathcal{M}(M_0)$.
Let $\mathfrak{p}=\mathbb{T}\setminus \mathcal{M}(M_0)$. Then $\mathfrak{p}\subseteq \mathfrak{m}(M')$.
By \eqref{circle2}, \begin{align*}\nu_Q^\ast(\boldsymbol{\varpi})
=\int_{\mathcal{M}(M_0)}S(\theta;\boldsymbol{\varpi})d\theta +\int_{\mathfrak{p}}S(\theta;\boldsymbol{\varpi})d\theta.\end{align*}
By Proposition \ref{propminor}, we have
\begin{align}\label{Salphavarpi4}
\int_{\mathfrak{p}}\big|S(\theta;\boldsymbol{\varpi})\big|d\theta
\leqslant  \int_{\mathfrak{m}(M')}\big|S(\theta;\boldsymbol{\varpi})\big|d\theta\ll \frac{W^2Y^{6}}{\phi(W)^{8}}M^{-1}.\end{align}
On choosing $M=\delta^{-8}c_0(Q)$ with some $c_0(Q)>0$ sufficiently large, we obtain from \eqref{nulower} and \eqref{Salphavarpi4} that
\begin{align*}\Big|\int_{\mathcal{M}(M_0)}S(\theta;\boldsymbol{\varpi})d\theta\Big|\gg \frac{\delta^8 W^2Y^{6}}{\phi(W)^8}.\end{align*}
This completes the proof of Proposition \ref{lemmadensityincrease} on applying Lemma \ref{lemmadensityincmajor}.

\section*{Acknowledgments}
This work is supported by the National Key Research and Development Program of China (Grant No. 2021YFA1000701) and the National Natural Science Foundation of China  (Grant No. 12471088).


\begin{thebibliography}{99}



\bibitem{BS} T. F. Bloom and  O. Sisask, {\it An improvement to the Kelley-Meka bounds on three-term arithmetic progressions},
arXiv:2309.02353.


\bibitem{BP} T. D. Browning and S. Prendiville, {\it A transference approach to a Roth-type theorem in the squares}, Int. Math. Res. Not.
IMRN.  {\bf 7} (2017), 2219--2248.


\bibitem{Chow} S. Chow, {\it Roth-Waring-Goldbach},  Int. Math. Res. Not. IMRN.
 {\bf 8} (2018), 2341--2374.

\bibitem{CLP} S. Chow, S. Lindqvist and S. Prendiville, {\it Rado's criterion over squares and higher powers}, J. Eur. Math. Soc. (JEMS)
 {\bf 23} (2021), 1925--1997.


\bibitem{Davenport} H. Davenport, 
{\it Analytic methods for Diophantine equations and Diophantine inequalities},
Cambridge University Press, Second Edition, 2005

\bibitem{Green2005} B. Green, {\it Roth's theorem in the primes},
Ann. of Math. (2) {\bf 161} (2005), 1609--1636.


\bibitem{Green}
B. Green,
{\it Quadratic forms in $8$ prime variables},
{Geom. Funct. Anal.} 35 (2025), 1587--1637.

\bibitem{Helfgott} H.~A. Helfgott, {\it Growth and expansion in algebraic groups over finite fields}, In: Analytic methods in arithmetic geometry, 71--111, Contemp. Math. Centre Rech. Math. Proc., 740, Amer. Math. Soc., 2019.

\bibitem{HH} K. Henriot and K. Hughes, {\it Restriction estimates of $\varepsilon$-removal type for $k$-th powers},
 Math. Ann. {\bf 372} (2018), 963--998.


\bibitem{Keil} E. Keil, {\it Translation invariant quadratic forms in dense sets},
arXiv:1308.6680.

\bibitem{Keilrefine} E. Keil, {\it Some refinements for translation invariant quadratic forms in dense sets},
arXiv:1408.1535.

\bibitem{KM} Z. Kelley and R. Meka, {\it Strong Bounds for $3$-Progressions}, In: 2023 IEEE 64th Annual Symposium on Foundations of Computer Science (FOCS 2023), 933–973, IEEE Computer Society, Los Alamitos, CA, 2023.
    
    \bibitem{Lang} S. Lang, {\it Algebra}, Graduate Texts in Mathematics, Vol. 211, Springer, 3rd edition, 2002.
        




\bibitem{Liu} J. Liu, {\it Integral points on quadrics with prime coordinates},
Monatsh. Math.  {\bf 164} (2011), 439--465.

 \bibitem{R1} K. F. Roth, {\it On certain sets of integers},
J. London Math. Soc. {\bf 28} (1953), 104--109.

    \bibitem{S16} J.-P. Serre, {\it Finite Groups: An Introduction. Surveys of Modern Mathematics}, Vol. 10, Int. Press, Somerville, 2016.
        
\bibitem{Zhao2016} L. Zhao, {\it The quadratic form in nine prime variables},
Nagoya Math. J. {\bf 223} (2016), 21--65.


\bibitem{Zhao2019} L. Zhao, {\it On translation invariant quadratic forms in dense sets},  Int. Math. Res. Not.
IMRN. {\bf 4} (2019), 961--1004.

\bibitem{Zhao2022}
L. Zhao,
{\it Translation invariant quadratic forms and dense sets of primes},
{Tran. Amer. Math. Soc.} 375 (2022), 725--752.


\end{thebibliography}
\end{document}